\documentclass[10pt]{amsart}
\usepackage{amsmath,amsfonts,amssymb,amsthm,graphicx,amscd}
\newtheorem{theorem}{Theorem}[section]
\newtheorem{lemma}[theorem]{Lemma}

\newtheorem{proposition}[theorem]{Proposition}
\theoremstyle{definition}
\newtheorem{definition}[theorem]{Definition}
\newtheorem{example}[theorem]{Example}

\newcommand{\dd}{{\rm d}}
\newtheorem{remark}[theorem]{Remark}
\numberwithin{equation}{section}
\input amssym.def
\input amssym
\begin{document}

\author{Vladimir   Bozin  and  Miodrag Mateljev\textrm{i\'{c}}}
%\address{Faculty of mathematics, Univ. of Belgrade, Studentski Trg 16, Belgrade, YU}
%\email{\rm miodrag@matf.bg.ac.rs}
%[Colip property of HQC] Quasiconformal harmonic mapping
\title{Quasiconformal  and HQC mappings between Lyapunov Jordan domains}
%%$C^{1,\alpha}$  V1 (in preparation)

%[2] [2010](2000)[2]
\subjclass[2010]{Primary 30C62, 31C05 .}
%30C80
\keywords{Harmonic maps, quasi-conformal maps, quasihyperbolic-metric, bi-Lipschitz maps,Lyapunov domains}

\thanks{Research partially supported by MNTRS, Serbia,  Grant No. 174 032}

\maketitle

\begin{abstract}
Let $h$  be    a quasiconformal  (qc) mapping of   the  unit disk  $\mathbb{U}$ onto  a Lyapunov domain.  We  show that $h$  maps subdomains     of     Lyapunov type of $\mathbb{U}$,  which touch the boundary of $\mathbb{U}$,
onto  domains of similar type.
In  particular  if $h$ is a harmonic qc  (hqc)  mapping of  $\mathbb{U}$  onto a Lyapunov  domain, using it,
we prove  that $h$ is co-Lipschitz (co-Lip) on  $\mathbb{U}$.
This settles an open intriguing problem.
%posed by Kalaj    in the subject  and can be regarded as  a version of Kellogg- Warschawski %theorem for hqc.
%%%%%

\end{abstract}
\section{Introduction}
%%%%%  lyp  {\rm lyp} (III)
Throughout the paper we consider the following setting  $(\mathbb{U}_{qc})$:   Let  $h:\mathbb{U}\rightarrow D$ be a $K$-qc map, where
$\mathbb{U}$ is the  unit disk   and suppose that  $D$ is a Lyapunov domain(see Definition \ref{dlyp1} below).
 If in addition    $h$  is  harmonic  we say that $h$ satisfies the hypothesis   ($\mathbb{U}_{hqc}$).
Under the  hypothesis  $(\mathbb{U}_{qc})$    we prove that    for every $a\in \mathbb{T}=\{|z|=1\}$,
there is a special Lyapunov domain $U_a$, of a fixed shape,   in the unit disk  $\mathbb{U}$  which touches $a$  and  a special, convex   Lyapunov   domain ${\rm lyp}(D)_b^-$(see
the subsection \ref{ss3.1}, the definition  (v)  before   the the proof  of  Theorem \ref{tmain1})\footnote{${\rm lyp}(D)_b^-= T_b (D_0^-)$, and  $D_0^-$ is defined    in the part (iii)  of  Proposition  3.5.}, of a fixed shape,   in $D$,     which  touches   $b=h(a)$, such that  ${\rm lyp}(D)_b^-\subset h(U_a)\subset H_b$, where  $H_b$  is    a half-plane  whose the boundary line contains $b$.
We can regard this result as ''a  good local approximation of a qc mapping $h$ by its restriction  to  a special Lyapunov   domain  so that its codomain is locally  convex".
In addition, if $h$ is harmonic, using this result,
%we first  prove    that for every $a\in \mathbb{T}$ there is  some %relative neighborhood   $V_a$ of $a$ such that  $h$ is  co-lipschitz on %$V_a$ and therefore
we prove  that $h$ is co-Lip  on  $\mathbb{U}$. This settles an open intriguing problem in the subject  and can be regarded as  a version of the Kellogg- Warschawski theorem for hqc.
%%%%%%
%%%%%%
In order to discuss the subject we first need a few basic definitions (see Section \ref{s_back} for more details).
\begin{definition}[Lyapunov  curves] \label{dlyp1}
(i) Throughout  the paper  by  $\varepsilon,\epsilon, c,c_1,\varepsilon_1,\epsilon_1, \kappa,\kappa_1$  etc.  we denote positive constants  and   by  $\mu,\mu_1$ etc.  constants in the interval $(0,1)$.

(ii)  Suppose that  $\gamma$ is a rectifiable, oriented, differentiable planar curve
given by its arc-length parameterization  $g$.  If
$$l_1={\rm lyp}(\gamma)={\rm lyp}(\gamma,\mu):=\sup_{t,s \in[0,l]}\frac{|g'(t)-g'(s)|}{|t-s|^{\mu}}<\infty,$$
we say that $\gamma$  is  a $C^{1,\mu}$  curve.
%a  Lyapunov curve  of order $\mu$ or $\mu$-Lyapunov) curve and we call  %$lyp(\Gamma)$,  Lyapunov  multiplicative constant.
$C^{1,\mu}$  curves are also known as Lyapunov (we say   also more precisely  $\mu$-Lyapunov) curves. We call  ${\rm lyp}(\gamma)$ the Lyapunov  multiplicative constant.
In this setting  we say that $\gamma$  is   $(\mu,l_1)$-Lyap (of order  $\mu$ with multiplicative constant $l_1$).
We say  that a    bounded planar  domain    $D$ is  $\mu$-Lyapunov   (respectively  $(\mu,l_1)$-Lyap), $0<\mu< 1$,  if it is bounded by   $\mu$-Lyapunov($(\mu,l_1)$-Lyap)  curve  $\gamma$.
In this setting it is convenient    occasionally    to   use  $l_1= l_1(D)$ instead of    ${\rm lyp}(\gamma)$.
\end{definition}

For a complex valued function defined on a domain in   the complex plane   $\mathbb{C}$,    we use the notation  $\lambda_f = l_f (z)= |\partial f (z)|   -
|\bar\partial f (z)|$ \, and   $ \Lambda_f (z)= |\partial f (z)|   +
|\bar\partial f (z)|,$   if $\partial f(z)$   and  $\bar\partial
f (z)$ exist.\\
%%%%%
%%%%%
%\begin{definition}[Special    domains  of Lyapunov type]
%For $\varepsilon,c> 0$  and $C|\varepsilon|^\mu<  \pi/2$, we use the notation  $L= L(\varepsilon)= {\rm Lyp}(\varepsilon,c)=\{w : c|w|^\mu<arg(w)<\pi-c|w|^\mu\,, |w|<\varepsilon\}$.
Note that    ${\rm Lyp}(\varepsilon,c)$ is  a special  domain of  Lyapunov   type   with two cusps  and  vertex  at $0$.
%\end{definition}
\begin{definition}[Elementary   Lyapunov  curves  and Special Lyapunov  domains]
%For  $c>0$, $0<\mu < 1$,  and $x_0 > 0$,      the curve $f(c,\mu)$  in the xy-plane is  defined  by\\
%(1)   $y=c |x|^{1+\mu}$,  $|x|< x_0$ .
The curve  $\gamma(c,\mu)=\gamma(c,\mu,r_0) $  is  defined, in polar coordinates $(r,\varphi)$, by joining the curves $\varphi= c r^{\mu}$  and  $\pi-
\varphi= c r^{\mu}$, $0\leq r < r_0$,  which share the origin(see Example \ref{ex1}  for more details).
%, has, near the origin, similar properties to the curve defined by (1).
An arc  $L$, which is     isometric   to the curve  $\gamma(c,\mu)$ we call an elementary   Lyapunov (more precisely $\mu$-Lyapunov)  curve. If $A$ is the isometry we call  $b=A(0)$ the vertex of $L$.  If an arc $C$ is a circle arc or  elementary   $\mu$-Lyapunov  for some  $ 0<\mu<1$  we call it an elementary   Lyapunov arc.

For $\varepsilon,c> 0$  and $c|\varepsilon|^\mu<  \pi/2$, we use the notation
\begin{itemize}
\item[(i)]
$L_0= L(\varepsilon)= {\rm Lyp}(\varepsilon,c,\mu)=\{w : c|w|^\mu<arg(w)<\pi-c|w|^\mu\,, |w|<\varepsilon\}$.
\end{itemize}
If this set is subset  of  $H$, $D$  it seems convenient to  denote it  shortly  by  $H_0$,  $D_0$ respectively.

%Roughly speaking,
A   special    domain of Lyapunov   type (with possible two cusps)
%smooth
is a  convex   domain whose the boundary consists of two  elementary   Lyapunov  curves. If  the part of boundary of a  Lyapunov ($\mu$-Lyapunov) domain is  an elementary   Lyapunov  curve with vertex at $b$, we call it  special  Lyapunov($\mu$-Lyapunov  with elementary arc)  domain with  vertex at $b$.
 \end{definition}
%defined by (1)
Note that the curve  $\gamma(c,\mu)$    is  $C^{1,\mu}$  but it is not  $C^{1,\mu_1}$  for   $\mu_1>\mu$ (at the origin), and
${\rm Lyp}(\varepsilon,c,\mu)$ is  a special  domain of  Lyapunov   type   with two cusps  and  vertex  at $0$.

As an application  of the Gehring-Osgood inequality\cite{GOs,Vu1} concerning  qc mappings  and quasi-hyperbolic distances,
in  the  particular case   of   punctured planes,  we  prove Proposition \ref{prop1}(we refer to this result as (GeOs)), which roughly stated says that: \\
if  $f$ is  a  $K$-qc mapping of the plane such that   $f(0)=0$,
$f(\infty)=\infty$   and   $z_1,z_2\in
\mathbb{C}^*$, then       the measures  of the
convex angles   between $f(z_1),f(z_2)$   and   $z_1,z_2$  can be compared. Using this we prove
the part  \rm{(IV)}  of  Theorem  \ref{thmain0}(we  shortly refer this result  as  (S-0)),  which  can be considered  as    our main result,   and  Theorem \ref{tmain1}  which  is a global version of  (S-0).

Theorem \ref{tmain1}  gives an approximation  of Lyapunov domains  by  special  Lyapunov domains and  it  is  a   crucial result  for the  application to hqc  mappings, stated here as: \begin{itemize}
%(i) (harmonic quasiconformal)
\item[(S-1)]   Suppose that $D$  is  a  Lyapunov   domain  and
  $h:\mathbb{U}\rightarrow D$ is a qc homeomorphism.
 Then   for every $a\in \mathbb{T}=\{|z|=1\}$,
  there is  a special Lyapunov   domain $U_a$, of a fixed shape,   in the unit disk  $\mathbb{U}$  which touches $a$  and  a special, convex   Lyapunov   domain ${\rm lyp}(D)_b^-$, of a fixed shape,   in $D$,    which  touches   $b=h(a)$, such that  ${\rm lyp}(D)_b^-\subset h(U_a)\subset H_b$, where  $H_b$  is    a half-plane  whose the boundary line contains $b$.  Using this  we reduce the proof of   co-Lip  property
  %(S1) to
\item[(L0)]: if   $h$ satisfies the hypothesis   ($\mathbb{U}_{hqc}$) then it is co-Lip,
 \end{itemize}
to what we call  locally convex case.
%(v) in the the proof of
 In order  to avoid confusion,  note that  in addition  Theorem \ref{tmain1} states  that  there is  a special, convex   Lyapunov   domain ${\rm lyp}(D)_b$(
 see Definition  \ref{dlyp2}), of a fixed shape,   in $D$, which  touches   $b=h(a)$, such that
 $h(U_a)\subset {\rm lyp}(D)_b \subset H_b$  (see figure  1). But we do not use this part in the proof  of (L0).
  %(ii)

  Set    $d_b(w)= dist (w,{\rm lyp}(D)_b^-)$,   and  for $a=h^{-1}(b)$, $d'_b(w)= dist (w, h(U_a))$.  By  an elementary  argument  one can prove:
\begin{itemize}
  \item[(L1)]  If  $w-b$ is in the direction of the normal vector $n_b$ of  $\partial D $  at $b$,   then
  %Lemma \ref{le_Lyp1},
$d_b(w)\approx |w-b|$  if $ |w-b|$  is  small enough.
\end{itemize}
%%%%%%
%%%%%%% bX2

Note  that   the subject of hqc mappings has been intensively studied by the participants of the Belgrade
Analysis Seminar (see Section \ref{s_back} for more details),
%\footnote{M. Pavlovi\'c, V. Markovi\'c, D. Kalaj, V. Bo\v zin,M. Arsenovi\'c, M. Markovi\'c, N. Laki\'c, D.  \v Sari\'c, M. Kne\v zevi\'c, V. Todor\v cevi\'c, M. Laudanovi\'c, M. Svetlik, I. Ani\'c,etc},cf. \cite{MMfilInv}
in particular by Kalaj, who  proved that if  $h$ is a hqc mapping of the unit disk onto  a  Lyapunov domain, then $h$ is   Lipschitz \cite{kamz} (Kalaj also probably first posed the problem of whether $h$ is, in fact,  bi-Lipschitz).
Since there is   a conformal mapping  of   the  unit disk  $\mathbb{U}$ onto  a $C^1$  domain  which is not  Lipschitz,  Kalaj's   result is nearly optimal.
In \cite{KK'_HQM2010},    it is shown  that a harmonic diffeomorphism $h$ between two $C^2$ Jordan domains is a ($K,K'$) quasiconformal mapping for some constants $K \geq 1$ and $K' \geq 0$ if and only if $h$ is bi-Lipschitz continuous (note that ($K,0$) qc is $K$-qc).
% An interesting problem  appears    in  the subject    concerning hqc %mapping:
These results naturally lead to  the following question (conjecture):

Question 1.   If $h:\mathbb{U}\rightarrow D$ is a hqc homeomorphism, where
$D$ is a  Lyapunov domain, is $h$ co-Lipschitz  (shortly co-Lip)?
%result. the

In  Theorem \ref{thmain2}  we give   an affirmative answer  to Question 1.

%Proposition  \ref{pro.nap1}
%\asymp
%Idea of proof of Theorem \ref{thmain2}.
%We also  use

The  following simple statements play an important role in the proof of Theorem \ref{thmain2} (co-Lip).
%%%%%%% eX2
%It  is convenient  to introduce the following hypothesis:
% first
%%%%%%bX3

%Proposition 5 \cite{napoc1}]
%\label{pro.nap1}
 Proposition 5 \cite{napoc1}  states that  if $h$ is a harmonic univalent orientation preserving $K$%
-qc mapping of domain $D$ onto $D^{\prime }$, then   $d(z)\Lambda _{h}(z)\approx d_{h}(z)$,  $z\in D$. We need only a corollary of this:
\begin{itemize}
\item[(S-2)]  $d(z)\Lambda _{h}(z)\succeq d_{h}(z)$,  $z\in D$.
\end{itemize}
%%%%%%

%Using   slightly modification
%of the proof of  Theorem 1.1  \cite{napoc1} (planar case)   we  show  %Theorem  \ref{thm.space1}.
%%%%%

Using   a slightly modification
of the proof of  Theorem 1.1  \cite{napoc1} (planar case)  and  Kellogg's theorem  we can  derive
%from
%Theorem \ref{thm.space1}  Proposition \ref {Harprop1}  stated here as:
%%%%%
\begin{itemize}
\item[(S-3)]    Suppose  that   $\,h\,$ is    a euclidean harmonic mapping
from a Lyapunov domain  $G$  into  a domain $D$   and  there is a half space $H_b$  which touches a  point  $b\in  \partial D$  such that   $D=h(G)\subset H_b$.
%the curve    \mathbb{B}
Then  $d(h(z),b) \succeq d_G(z)$,    $z \in G$.
\end{itemize}
   We    say that a domain D  is locally convex  at  a point   $b\in  \partial D$  if there is  a  half  space $H_b$  such that  $D\subset H_b$. %

For the convenience of the reader we summarize     that   (S-1),(S-2)  and  (S-3),   are the main ingredients in the proof   of    Theorem \ref{thmain2}   stated here  as
\begin{theorem}\label{thmain2}
Suppose $h:\mathbb{U}\rightarrow D$ is a hqc homeomorphism, where
$D$ is a  Lyapunov domain with $C^{1,\mu}$ boundary. Then $h$ is co-Lipschitz.
\end{theorem}
\begin{remark}
Note that, in general,  $h(U_a)$ is not convex  and  we can not  apply  %Theorem \ref{t.in.belowA} $\rm{(ii)}$
our consideration \cite{napoc1} (see the proof of Theorem 1.2 there)
%\cite{revroum01}
directly; but  $h(U_a)\subset H_b$   is locally  convex  at $b$  and we can apply (S3)
(note that  we do not use the fact   that   ${\rm lyp}(D)_b^-$  is  a  convex  Lyapunov domain).
% $\mathcal{D}_1$ $\hat{D}_0$
\end{remark}
  Recall  that a  mapping $h$  which is ($\mathbb{U}_{qc}$) satisfies  (S-1). If $h$ is in addition harmonic  then we can apply  (S-3). This is crucial for the proof  of theorem  and it reduces the proof to the locally-convex case.

%, i.e. belonging to   $\mathcal{D}_1$
%Kalaj \cite{kamz} proved that  $h$ is  Lipschitz.
%%%%%%eX3

\begin{definition}[Hypothesis  ($\mathbb{H}_{qc}$),($\mathbb{H}_{hqc}$), ($\mathbb{H}_{qc}^0$), ($\rm{H}$-0),($\mathbb{U}_{qc}$),   ($\mathbb{U}_{hqc}$)  and ($\rm{U}$-1)]
\label{1.3}
It is convenient to consider the following  hypothesis
\begin{itemize}
%that
\item [(Sp0):]   $0 \in \partial D$ and  $D$ has the real axis as a tangent at
$0$, with inner normal pointing upwards.
\item [($\mathbb{H}_{qc}$):]  Let  $h:\mathbb{H}\rightarrow D$ be   $K$-qc map, where
$\mathbb{H}$ is the upper -half plane  and suppose $D$ is a Lyapunov domain
%Jordan domain with $C^{1,\mu}$
with the boundary  boundary $\partial D$   positively oriented.
\end{itemize}
Using rotation  and translation,   we can reduce the study  of the behaviour  of $h$ at a point $a\in \mathbb{R}$   to the following setting
%quote
\begin{itemize}
\item [(Lyp-0):]  $D$  is a Lyapunov domain  and  satisfies (Sp0).
%Suppose   that $0 \in  \partial D$ and that $D$ has the real axis as a tangent at  $0$, with inner normal pointing upwards.
\end{itemize}
%{\bf
If, in addition to   ($\mathbb{H}_{qc}$),      $h$ is harmonic,  we say that  $h$ satisfies the hypothesis   ($\mathbb{H}_{hqc}$).\\
If $h$ satisfies the hypothesis   ($\mathbb{H}_{qc}$) (respectively   ($\mathbb{H}_{hqc}$)),   $D$ satisfies the hypothesis (Lyp-0)  and  $h(0)=0$, we say that  $h$ satisfies the hypothesis
%(${qc}(\mathbb{H}, D;Lyp,0$), or  shortly
($\mathbb{H}_{qc}^0$),   (respectively    ($\rm{H}$-0)).\\
If  $\mathbb{H}$ is replaced by   $\mathbb{U}$ we denote  the corresponding hypotheses by  ($\mathbb{U}_{qc}$), ($\mathbb{U}_{hqc}$); and $\mathbb{U}_{qc}^0$ (respectively $\rm{U}$-1) if in addition $h(1)=0$.
%%%%%$h(0)=0$
\end{definition}
Note that in this paper we consider only the planar case.
The  plan of the   exposition is as follows:
In Section   \ref{s_back},   we consider the   background, definitions    and basic properties   of Lyapunov  domains   and we   prove     Proposition \ref{prop1}, which may be considered to be a version of the  Gehring-Osgood inequality  related to the measures of the
corresponding angles.
 %Version 1 of
In Section \ref{sec2},   we prove   Theorem  \ref{thmain0} and  Theorem \ref{tmain1}.
%, and   outline a   proof of    Theorem   \ref{thmain2}(co-Lip).
In  Section   \ref{s_co_lip}  we give
%Version 2 of
the proof of    Theorem   \ref{thmain2}(co-Lip).

The second author   communicated  the main result of this paper at CMFT 2017.\footnote{see  Cmft2017, Jule 10-15,
Lublin, Poland (see http://cmft2017.umcs.lublin.pl), plenary speakers.}
%as a plenary lecture

We also suggest  to the interested  reader to make rough picture and scheme with corresponding  notations in order to follow the manuscript;  and first  to read Section  3   without proofs  and and then Section 4 with all  details and finally  to consider complete proofs  and  technical details  in Section 3.

\section{Background}\label{s_back}
%\bigsqcup
%%%%%%%%%%%%%%%%%%%%%%%%%%%%%%% {rckm-fil}

The next example which is shortly discussed in \cite{gol,rckm-fil},  shows that  there is  a conformal map of unit disk onto $C^1$  domain which is not    bi-Lipschitz.
\begin{example}\label{e_nLip}
Set
$$ w=A(z)= \frac{z}{\ln \frac{1}{z}},\quad w(0)=0 \,.$$
Note  $\ln \frac{1}{z}= - \ln z$,    $w'(z)=- (\ln z)^{-1} + (\ln
z)^{-2}$  and  $w'(z) \rightarrow 0$  if  $z\rightarrow 0$
throughout $\mathbb{H}$.
For $r$
small enough $A$ is univalent in $U_r^+$. We can check that there
is a smooth domain $D\subset U_r^+$ such that  interval $(-r_0,r_0)$, $r_0
>0$,  is a part of the boundary  of $D$,    $D^*= A(D)$  is $C^1$
domain  and $A$ is not co-Lipschitz on $D$.
\end{example}
%%%%%%%%%%%%%%%%%%%%%%%%%%%%%%%
For basic properties of qc  mappings  the reader can consult   Ahlfors's lovely book
\cite{Ahl2006}.
Let $\gamma$ be a Jordan curve. By the Riemann mapping theorem there exists a
Riemann conformal mapping of the unit disk onto the Jordan domain $G = {\rm int} \gamma$. By
Caratheodory's theorem it has a continuous extension to the boundary. Moreover, if
$\gamma \in C^{n,\alpha}$, $n\in \mathbb{N},\,\, 0\leq \alpha < 1$, then the Riemann conformal mapping has a $C^{n,\alpha}$  extension
to the boundary (this result is known as Kellogg's theorem), see \cite{war}.
Conformal  mappings are quasiconformal and harmonic. Hence quasiconformal harmonic
(abbreviated by HQC) mappings are a natural generalization of conformal  mappings.
%%%%%%%%%
\begin{remark} Note that\\
a) The proof of  Kellogg's theorem for conformal mapping  is not elementary and it is based on some techniques which we can not adapt for hqc.\\
b)Since  there is  a conformal map of unit disk onto $C^1$  domain which is not    bi-Lipschitz (Example \ref{e_nLip} above), it seems  that the hypothesis  that domains are Lyapunov is essential.\\
By a) and b) in mind,  it seems that   we need new approaches  to study  hqc  mappings.
\end{remark}
%%%%%%%%%%%%%%%%%%%%%%
Recall that HQC
%harmonic quasiconformal
%(abbreviated by HQC)
mappings are now a very active area of investigation and some  new methods have been  developed for studying this subject (see for example \cite{MMfilInv} and literature cited there).
%[22]
Concerning the background we mention only a few  results which are closely related  to our results:\\
It seems  that   O. Martio \cite{Om}  was the first one  who  considered
%observed  that  a
HQC mapping of the unit disk
%is co-lip
and M. Pavlovi\'c  proved in  \cite{MP}  that it is  Lipschitz.
An  asymptotically  sharp variant have been obtained by Partuka and Sakan  \cite{part0}.
Among other things Kne\v zevi\'c and the second
author in \cite{KnMa} showed that a $K$-qc  harmonic mapping of the unit disk onto itself is a
$(1/K, K)$   quasi-isometry with respect to the Poincar\'{e} and Euclidean metrics.  For   bi-lipschitz approximations  of quasiconformal  maps see     Bishop \cite{bish}.
M.  Mateljevi\'c  \cite{napoc1}  and   V. Manojlovi\'c  \cite{man} showed  that hqc  mappings are  Bi-Lipschitz with respect to  quasi hyperbolic metrics.
Since the composition of a harmonic mapping and a conformal mapping is itself
harmonic, using the case of the unit disk and Kellogg's theorem, these theorems  can be generalized to the class of mappings from arbitrary Jordan domains with  Lyapunov  boundary onto the unit disk. However the composition of a conformal and a  harmonic mapping is not, in general, a harmonic mapping. This means in particular,  that results of this kind for arbitrary image domains do not follow directly  from the case in which the codomain is   the unit disk or the upper half-plane and Kellogg's theorem.
%%%%%
%In particular,in [15] we
In \cite{km.anal06}, Kalaj and the second author  show how to combine  Kellogg's theorem with the so called inner type estimate and that the  simple proof in the case
of the upper half-plane has an analogue for $C^2$ domains; namely, they
 proved  a version of the "inner estimate" for quasi-conformal diffeomorphisms, which  satisfies  a  certain
%M. Mateljevic´ / Filomat 29:9 (2015), 1953–1967 1959
estimate concerning their Laplacian. As an application of this estimate, it is  shown  that quasi-conformal harmonic mappings between smooth domains (with respect to the approximately analytic metric), have bounded partial derivatives; in particular, these mappings are Lipschitz. The  discussion in \cite{km.anal06} includes harmonic mappings with respect to (a) spherical  and Euclidean  metrics (which are approximately analytic) as well as (b) the metric induced by the holomorphic quadratic differential.
%%%%% b Ins 1\\

%Although the evidence of the  following two  statements are relatively
%  %%%%%RCK
Although  the  following two  statements
did not get attention immediately after their publications,
%they played
it turns out, surprisingly,  that  they play an important role in the proof of Theorem \ref{thmain2} (co-Lip).
\begin{proposition}[Corollary 1,  Proposition 5
\cite{napoc1}; see also  \cite{man}]
\label{c.gd} \textit{Every e-harmonic quasi-conformal mapping of
the unit disc (more generally of a strongly hyperbolic domain) is
a quasi-isometry with respect to  the hyperbolic distance}.
\end{proposition}
%VMfil2009  \cite{napoc1,man} (Proposition \ref{c.gd} here)
%See also  \cite{man}  and
\begin{theorem}[\cite{revroum01}]\label{t.in.belowA}
%$(ii.1)$
Suppose  that   $\,h\,= f + \overline{g}$ is    a Euclidean
orientation preserving  harmonic mapping from $\mathbb{U}$ onto the
bounded convex  domain $D=h(\mathbb{U})$, which contains a disc
$\,B(h(0);R_0)\,$.
\begin{itemize}
\item[(I)]  Then $ |f'| \geq\ R_0/4$ on     $\mathbb{U}$.
\item[(II)]Suppose, in addition,  that  $h$  is qc.  Then  $l_h\geq (1-k)
|f'|\geq (1-k)R_0/4$  on  \,$\mathbb{U}$.
\item[(III)]  In particular,    $h^{-1}$ is Lipschitz.
\end{itemize}
\end{theorem}
See  also Partyka and Sakan \cite{part1}.

Concerning the Lipschitz property  of   hqc,  Kalaj \cite{kamz} proved:
%$h$ is  Lipschitz. , i.e. belonging to$\mathcal{D}_1$
\begin{theorem}\label{thmainDK}
Suppose $h:D_1\rightarrow D_2$ is a hqc homeomorphism, where $D_1$
and $D_2$ are domains with $C^{1,\mu}$ boundary.
\begin{itemize}
\item[(I)] Then $h$ is Lipschitz.\\
\item[(II)]  If, in addition, $D_2$ is convex, then $h$ is bi-Lipschitz.
\end{itemize}
\end{theorem}
With this theorem in mind  Question 1 is natural.
 The proof of part (a) of Theorem \ref{thmainDK} in \cite{kamz} is based on an application of Mori's theorem on quasiconformal mappings, which  has also been used in \cite{MP} in the case $D_1=D_2=\mathbb{U}$,  and a geometric lemma related to Lyapunov domains.
%%%%%%%%   Rel result

\subsection{Notation}
%%%%
%In order to discuss the subject we need first
Here we give  a few basic definitions.
\begin{definition}[qc]
(i) By   $\mathbb{C}$  we denote the  the complex plane  and by  $\mathbb{T}$  the unit circle.  For $r>0$ and $w \in \mathbb{C}$,  we denote  by   $B(w,r)$ and  the  $C(w,r)$  the disk and  circle of radius $r$ with center at $w$.

(ii)  By   $\mathbb{C}^*$  we denote the punctured  complex plane  $\mathbb{C}\setminus \{0\}$,    by  $\mathbb{H}^*$   the lower half plane  $\{z: {\rm Im} z<0\}$ and by  $\mathbb{U}^+$  the upper  half disk  $\{z: {\rm Im} z> 0, |z|<1\}$.

(iii)  Recall that,  for a complex valued function $h$  defined on a domain in   the complex plane   $\mathbb{C}$,    we use the notation
 $$\lambda_h = l_h (z)= |\partial h(z)|-|\bar\partial h(z)| \quad \mbox{and}\quad   \Lambda_h (z)= |\partial h (z)|   +
|\bar\partial h (z)|,$$   if $\partial h(z)$   and  $\bar\partial
h (z)$ exist.
%%%%%
A  homeomorphism  $h \colon D\rightarrow G, $ where $D$ and $G$ are
subdomains of the complex plane $\mathbb C,$ is said to be
$K$-quasiconformal ($K$-qc  or  $k$-qc), $K\ge 1$, if $f$ is
absolutely continuous on a.e. horizontal and a.e. vertical line in
$D$  and  there is  $k \in[0,1)$  such that

\begin{equation}\label{defqc0}
|h_{\bar z}|\le k|h_z|\quad \text{a.e. on $D$},
\end{equation}
where $K=\frac{1+k}{1-k}$,    i.e. $k=\frac{K-1}{K+1}$.

Note that the condition (\ref{defqc0}) can be written as

\begin{equation}
D_{h}:=\frac{\Lambda_h}{\lambda_h } = \frac{ |h_{z} | + |
h_{\overline{z}} | }{| h_{z}| - |h_{\overline{z}}|} \leq K,
\end{equation}
where $K=\frac{1+k}{1-k}$,  i.e. $k=\frac{K-1}{K+1}$.

(iv)   Let $\Omega \subset \mathbb{R}^{n}$ and $\mathbb{R}^{+}=[0,\ \infty )$ and $f,\
g:\Omega \rightarrow \mathbb{R}^{+}$. If there is a positive constant $c$
such that $f(x)\leq c\,g(x)\,,\ x\in \Omega \,$, we write $f\preceq g$ on $%
\Omega $. If there is a positive constant $c$ such that

\begin{equation*}
\frac{1}{c}\,g(x)\leq f(x)\leq c\,g(x)\,,\quad x\in \Omega \,,
\end{equation*}%
we write $f\approx g$ (or $f\approx g$ ) on $\Omega $.
\end{definition}
 To gain some intuition about  Lyapunov curves   we give a basic   example:
%the curve $f(c,\mu)$ For  $c>0$, $0<\mu < 1$,  and $x_0 > 0$ the curve $f(c,\mu)$  in the xy-plane is  defined  by ,  $|x|< x_0$,
\begin{example}\label{ex1}
For  $c>0$, $0<\mu < 1$,  and $x_0 > 0$,  the curve $f(c,\mu)=f(c,\mu,x_0)$  in the xy-plane which  is  defined  by\\
 (1)   \quad                   \quad       \quad   \quad    $y=c |x|^{1+\mu}$,  $|x|< x_0$,\\
is  $C^{1,\mu}$ at the origin but is not  $C^{1,\mu_1}$  for    $\mu_1  >\mu$.  It is convenient to write  this equation using polar
coordinates  $z= r e^{i \varphi}$  in the form:  $r \sin\varphi= c r^{1+\mu} (\cos \varphi)^{1+\mu}$.
 Next, if  $0\leq \varphi \leq \pi/2$,  we have
$\sin\varphi= c r^{\mu} (\cos \varphi)^{1+\mu}$, $0\leq r < r_0$, where $r_0$ is a positive number.
Since   $\sin\varphi= \varphi +o(\varphi)$  and  $\cos \varphi=1
+ o(1)$,   we  find  $\varphi= c r^{\mu}+o(1)$    when  $\varphi\rightarrow 0$.
%%%%%%%%%%%%%%%%%%%%%%%%%%%%%%%%%%%%%%%%%%%%%%%%%%%%%%%%%%%%%%
If    $\pi/2 \leq \varphi \leq \pi$,  we  have
$\sin(\pi-  \varphi) =\sin\varphi= c r^{\mu} (\cos \varphi)^{1+\mu}$, $0\leq r < r_0$, where $r_0$ is a positive number.
Since   $\sin (\pi- \varphi)= \pi- \varphi +o(\pi- \varphi)$  and  $\cos \varphi=-1
+ o(1)$,  we find    $\pi-\varphi= c r^{\mu}+o(1)$      when  $\varphi\rightarrow \pi$.
%%%%%%%%%%%%%%%%%%%%%%%%%%%%%%%%%%%%%%%%
The curve $\gamma(c,\mu)=\gamma(c,\mu,r_0)$
defined by joining the curves $\varphi= c r^{\mu}$  and  $\pi-
\varphi= c r^{\mu}$, $0\leq r < r_0$,  which share the origin,
has similar properties  near the origin to the curve defined by
(1). The reader can check that the curves $f(c,\mu)$   and    $\gamma(c,\mu)$  are   $C^{1,\mu}$
at the origin but are  not  $C^{1,\mu_1}$  for $\mu_1 >\mu$.
%We can

Note that if a curve satisfies $\varphi \leq c r^{\mu}$, then it
is  is below the curve $\gamma(c,\mu)$.
\end{example}
\subsection{Gehring-Osgood inequality}
We can compute the quasihyperbolic metric $k$ on $\mathbb{C}^*$ by
using the covering $\exp : \mathbb{C}\rightarrow \mathbb{C}^*$,
where $ \exp$ is the exponential function. Let $z_1,z_2\in
\mathbb{C}^*$, $z_1= r_1 e^{it_1}, z_2= r_2 e^{it_2} $  and
% Then
$\theta =\theta(z_1,z_2) \in [0,\pi]$ the measure of the convex angle
between $z_1,z_2.$ We use
%will prove
$$k(z_1,z_2)= \sqrt{\Big|\ln \frac{r_2}{r_1}\Big|^2 + \theta
^2} .$$ This well-known formula is due to Martin and Osgood.
%, see \cite[(3.12), p. 36]{vaisala}.

Let  $\ell=\ell(z_1)$ be the line  defined by  $0$ and $z_1.$  Then $z_2$
belongs to one half-plane, say $M$,  on which $\ell=\ell(z_1)$ divides
$\mathbb{C}.$

Locally, denote by $\ln$ a branch of $\mathrm{Log}$ on $M.$ Note
that $\ln$ maps $M$ conformally   onto a horizontal strip of width
$\pi.$ Since $w=\ln z$, we find that the quasi-hyperbolic  metric
$$|\dd w|= \frac{|\dd z|}{|z|}.$$
Note that  $\rho(z)= \frac{1}{|z|}$ is the
quasi-hyperbolic  density for $z\in \mathbb{C}^*$ and therefore
$$k(z_1,z_2)= |w_1 - w_2|=|\ln z_1 - \ln z_2|.$$
Let $z_1,z_2\in \mathbb{C}^*$,   $w_1= \ln z_1= \ln r_1 + i t_1.$
Then   $z_1= r_1 e^{it_1}$ and there is $t_2\in [t_1, t_1 +\pi)$ or
$t_2\in [t_1-\pi, t_1 )$  such that  $w_2= \ln z_2= \ln r_2 + i t_2$ .
Hence
$$k(z_1,z_2)= \sqrt{\Big|\ln \frac{r_2}{r_1}\Big|^2 + (t_2 -t_1 )^2} \,.$$
%and therefore  this
Now  using the  quasi-hyperbolic distance $k$    as a corollary of the Gehring-Osgood inequality, we can prove the following result which we will need.

\begin{proposition}\label{prop1}
Let  $f$ be a  $K$-qc mapping of the plane such that   $f(0)=0$,
$f(\infty)=\infty$ and   $\alpha=K^{-1}$\,.   If $z_1,z_2\in
\mathbb{C}^*$, $|z_1|=|z_2| $ and $\theta \in [0,\pi]$
\rm{(}respectively $\theta^*\in [0,\pi]$\rm{)} is  the measure of
the convex angle between $z_1,z_2$ \rm{(}respectively
$f(z_1),f(z_2)$\rm{)}, then

$$ \theta^* \leq c \max \{\theta^{\alpha},
\theta\},
$$
where $c=c(K).$ In particular, if $\theta\leq 1$, then  $\theta^*
\leq c \theta^{\alpha}.$
\end{proposition}
\begin{proof}
By the  Gehring-Osgood inequality,

$$ k\Big(f(z_1),f(z_2)\Big) \leq c \max \{k(z_1,z_2)^{\alpha},
k(z_1,z_2)\},
$$  where $c=c(K).$
It is clear  that   $\theta^*\leq k\Big(f(z_1),f(z_2)\Big)$.   Since  $|z_1|=|z_2|$ and $k(z_1,z_2)=t_2 -t_1=\theta$, we get the desired result.
\end{proof}

\section{Main result}\label{sec2}
We first need some definitions.\\
Elementary  Lyapunov  domains, Arc-chord constant $b_\gamma$  and  the  second Lyapunov constant  $l^2_D= l^2_\gamma$.

\begin{definition}[Elementary   Lyapunov  domains]
(i)   Recall  for $r>0$ and $w \in \mathbb{C}$,  we denote  by   $B(w,r)$ and  the  $C(w,r)$  the disk and  circle of radius $r$ with center at $w$.
% and by  $C^+(r)$ the half circle in the upper half  plane.
In particular,  we use notation  $B(r)$ and  the  $C(r)$  for  the disk and  circle of radius $r$ with center at $0$ and we denote  by  $C^+(r)$ the half circle in the upper half  plane.

(ii)Definition of   $L_b^-(\varepsilon)$.  Further for   $v>0$   let  the circle $C(iv,r)$ touch the curve   $\gamma=\gamma(\mu,c)$  at points $w_1$ and $w_2$ (say that  $ u_1<u_2$)  and let $l^+$ be  the upper half arc   of the circle $C(iv,r)$
joining   $w_1$ and $w_2$   and  $\gamma_1$ be the part of $\gamma$ over
$[u_1,u_2]$,  where    $u_k= {\rm Re}  w_k$, $k=1,2$. Then the domain enclosed by     $l^+$  and  $\gamma_1$  we denote by ${\rm Lyp} (r,c)$. If $\epsilon^0 $  is maximum of $r>0$  for which
${\rm Lyp} (r,c)$  belongs to  $L(\varepsilon)$, we   denote  the domain   ${\rm Lyp} (\epsilon^0,c)$  by   ${\rm Lyp}^- (\varepsilon,c)$. If A  is an euclidean isometry  and $A(0)=b$, we denote the domain  $A({\rm Lyp}^- (\varepsilon,c))$  by $L_b^-(\varepsilon)={\rm Lyp}_b^- (\varepsilon,c)$  and call it an elementary  $\mu$- Lyapunov  domain.
\end{definition}
Although the boundary  of  an elementary  Lyapunov  domain consists  of an elementary  $\mu$- Lyapunov arc $\gamma_0$  and a circle arc $C_0$  with common end points, say  $a_0$ and $b_0$    note that  it   has no cusps  because  $\gamma_0$  and $C_0$  have common tangents  at  points $a_0$ and $b_0$.
 \begin{definition}[arc-chord condition]
 More generally, if we suppose only  that  the curve is rectifiable
 we can define  the distance along it. Let  $C$ be a rectifiable
 Jordan closed curve  and $z_1,z_2$ finite points of $C.$ They
 divide $C$ into two arc, and we consider one with smaller
 Euclidean length and denote its length with  $d_C(z_1,z_2).$
 \begin{itemize}
 \item[(a)]  The
 curve $C$  is said to satisfy the arc-chord condition if the ratio
 of this length to the distance $|z_1-z_2|$ is bounded by a fixed
 number $b_C=b_C^{arc}$ (which we call  arc-chord constant of $C$)   for all finite $z_1,z_2\in C.$%%%%%
\item[(b)]
 The curve $C$  is said to satisfy the arc-chord condition at a
 %The proof  of  the part  \rm{(IV)}  of  Theorem  \ref{thmain0},which is our main result,  is based on   that  a $C^1$ curve   satisfies  the arc-chord condition.
 fixed point $z_1\in C$  if the ratio of the  length $d_C(z_1,z)$
 to the distance $|z_1-z|$ is bounded by a fixed number   $b_C(z_1)=b_C^{arc}(z_1)$ for all finite $z\in C.$
 %We will prove We will prove   $b_C(z_1)=b_C^{arc}(z_1)$ for all finite $z\in C.$
 \end{itemize}

\begin{itemize}
\item[(c)]
If  $D$ is  a  $\mu$-Lyapunov  domain bounded by a curve $\gamma$, we define     $l_2=l_2(D)=l^2_D= l^2_\gamma=\frac{\pi}{2}\,l_1\, b_\gamma^{1+\mu}$, and we call it the  second Lyp-constant,  where  $l_1=lyp(\gamma,\mu)$.
\end{itemize}
\end{definition}

 %Arc-chord constant $b_\gamma$  and  second Lyapunov constant  $C^2_D= C^2_\gamma$
 \subsection{Auxiliary  results}\label{ssAux}
Suppose that $D$ satisfies the hypothesis  (Lyp-0).
Further, one can prove
%without  loss  of  generality  we can suppose that
\begin{itemize}
\item[(c1)] It is known that  that  a $C^1$ curve   satisfies  the arc-chord condition.
\item[(d)] there is  $r_1>0$ such $\partial D \cap B(r_1)$ is graph of a  function $F$,   $v=F(u)$,  $-u_1 <u < u_2$, where  $u_1,u_2>0$, and that
  the set  $V=\{ (u,v):-u_1 <u < u_2, F(u)<v \}\cap B(r_1)$ belongs $D$, where $u,v$  are the cartesian coordinates in w-plane  and $w=u+iv$.
%$\epsilon$ such that  on the circle  $C(\epsilon)$
 \item[(e)] Let $D$ be  a  bounded   Lyapunov  domain,    $a_0\in D$   and  let $\psi $  be a conformal mapping of $D$ onto $\mathbb{U}$  with  $\psi(a_0)=0$. Then
there are constants $k_1=\underline{k}_1(D,a_0)$ and $k_2=\underline{k}_2(D,a_0)$ (which we  call   the lower and upper   Kellogg multiplicative constants
of $D$  with respect to $a_0$  respectively)  such that    $k_1 |z_1-z_2|\leq |\psi (z_1)-\psi (z_2)| \leq k_2 |z_1-z_2|$,  $z_1,z_2\in D$.
\item[(f)]  Mori's theorem.   Let  $f:\mathbb{U}\rightarrow  \mathbb{U}$  be a surjective K-qc  mapping with $f(0)=0$  and  $\alpha=1/K$. Then
$|f(z_1)- f(z_2)|  \leq 16 |z_1- z_2|^{\alpha}$, $z_1,z_2\in \mathbb{U}$,
i.e.  $f$ is $\alpha$-Holder continuous.

\item[(g)] Let the mapping $A$ is given  by $A(z)= i \frac{z-i}{z+i} + i$.  Then  $A(i)=i$  and  $A$  maps  $\mathbb{H}$  onto $B_1=B(i,1)$.
Since  $A'(z)= \frac{2}{|z+i|^2}$, we first find   $|A'(z)|\leq 2$  and therefore     $|A(z)|\leq 2|z|$,  $z\in \mathbb{H}$.

%??  $l_0= 16 2^{\alpha} k_1^{-1}$
%$B(z)=  i \frac{1-z}{1+z}$
%Set  $\breve{h}= h\circ B$  apply   proposition \ref{prop1} to  $\breve{h}$ there is  a disk  $B(0,r_0)$ such  that  for $h$
\item[(h)]  Suppose that  $D$ is a bounded convex  planar domain,     $f: \overline{D} \rightarrow \mathbb{C}$ is holomorphic  mapping  and $z_0\in D$. Then there is a constant
%o  $\overline(D)$
$c>0$   such that  $|f z-f z_0| \leq c |z-z_0|$, $z\in D$. The proof is straightforward.
\item[(i)]  Let  $f$ be a  $K$-qc mapping of the half -plane $\mathbb{H}$  on a domain $D$ such that   $f(0)=0$, and suppose that $\partial D$ is  a  $K$-quasi-circle
%$f(\infty)=\infty$
and   $\alpha=K^{-1}$.  Then  $f$  has a $K_1$-qc extension to a map $\tilde{f}$    of the complex
plane,  which by abuse of notation we denote sometimes  again by $f$  if there is no possibility of confusion.
\end{itemize}
\begin{definition}\hfill
\begin{itemize}
\item[(i)] If $z_1,z_2\in \mathbb{C}^*$   by   $\theta(z_1,z_2)$ we denote   the measure of
the convex angle between $z_1,z_2$.
\item[(ii)]  For $p\in \mathbb{C}$,  set
$$X(z)=X_{p}(z)=\frac{p z}{z- p}, z \in \overline{\mathbb{C}}\quad  \mbox{and}\quad  Y=Y_p= X^{-1}.$$
\item[(iii)] If $f$ is homeomorphism of $\overline{\mathbb{C}}$  onto itself, we define $p=p(f)= f^{-1}(\infty)$.
 \item[(iv)]  If $\gamma$ is an arc in  $\mathbb{C}$   and  $Z:\gamma\rightarrow  \mathbb{C}^*$ continuous map  by   $\Delta_\gamma {\rm Arg} Z$  we denote the variation  of
${\rm Arg} Z$   along $\gamma$.
\end{itemize}
\end{definition}
Note that $X$ and $Y$ are M\"{o}bius automorphisms of $\overline{\mathbb{C}}$ with the following properties: $Y(z)=-\frac{p z}{z- p}$, $X(0)=Y(0)=0$, $X(p)=\infty$,
$X(\infty)=p$,     $Y(p)=\infty$   and   $Y(\infty)=-p$.
If we set  $\breve{f}= f\circ X$, then  $f=\breve{f}\circ Y$.
   $X_p$  and  $Y_p$ map lines   $l_\beta =\{ r e^{i \beta}: r\in\mathbb{R} \}$ onto the circles which contain $0$ and $p$.
Since $Y_p$ map   the circle  $C(0,|p|)$  onto line $L$ which does not contain $0$.
%If  $p=|p|e^{i \varphi_0}$,
If  $z_n= e^{-i/n}p$   and  $z_n'= e^{i/n}p$, then
$\theta(z_n',z_n)\rightarrow 0$  and  $\theta(Xz_n',Xz_n)\rightarrow \theta_0$, $\theta_0\neq 0$, if $n\rightarrow \infty$.
This example shows  that  we need to adapt  a version of  Proposition  \ref{prop1}   to hold for the mappings  $X_p$.
%\ref{prop1} Proposition  2.8
\begin{proposition}\label{prop0}
Let  $f$ be a  $K$-qc mapping of the plane $\overline{\mathbb{C}}$  onto  itself,   $f(0)=0$,  $p=f^{-1}(\infty)$,
%a domain $D$ such that   $f(0)=0$,
%$bD$ quasi-circle
%$f(\infty)=\infty$  \cap B(0,r_0)
$\alpha=K^{-1}$\,  and  $r_0=|p|/2$.
\begin{itemize}
\item[(I)] {\rm (a)} Then    $f= \breve{f} \circ  Y$,  where  $Y=Y_{p}$, $\breve{f}$  is $K$-qc mapping of the plane $\overline{\mathbb{C}}$  onto  itself,  with   $\breve{f}(0)=0$ and      $\breve{f}(\infty)=\infty$.\\
%\item[(a)]
{\rm(b)} If $z_1,z_2\in
\mathbb{C}^*$, $|z_1|=|z_2|$ and $\theta \in [0,\pi]$
\rm{(}respectively $\theta^*\in [0,\pi]$\rm{)} is  the measure of
the convex angle between $z_1,z_2$ \rm{(}respectively
$\breve{f}(z_1),\breve{f}(z_2)$\rm{)}, then $ \theta^* \leq c \max \{\theta^{\alpha},
\theta\},$
where $c=c(K).$ In particular, if $\theta\leq 1$, then  $\theta^*
\leq c \theta^{\alpha}.$
\item[(II)]  If $z_1,z_2\in \mathbb{C}^*\cap B(0,r_0)$, $|z_1|=|z_2|$,  then
   $$\theta(Xz_1,Xz_2)\leq (1+ r_0^{-1}) \theta(z_1,z_2).$$
%In particular,  if $z\in \mathbb{C}^*\cap B(0,r_0)$, $r=|z|$,
\item[(III)]  For given  $H_0'={\rm Lyp}(\varepsilon,c,\mu)$,  $\varepsilon< r_0$, there is $H_0={\rm Lyp}(\varepsilon_1,c_1,\mu_1)$   such   that  $Y(H_0) \subset H_0'$.

%%%%%%%%%%%%%%%%%%%%%%%%%%%%%%%%%%%%%%%%%%%%%%%%%%%%%%%%%%%%%%%%%%%%%%%
%%%%%%%%%%%%%%%%%%%%%%%%%%%%%%%%%%%%%%%%%%%%%%%%%%%%%%%%%%%%%%%%%%%%%%%%%%

\end{itemize}
\end{proposition}
\begin{proof}(I) Set  $\breve{f}= f\circ X$. Since  $X(p)=\infty$  and $p=f^{-1}(\infty)$, we  have  $\breve{f}(\infty)=\infty$.
Since $X$  is   M\"{o}bius automorphism of $\overline{\mathbb{C}}$,  $\breve{f}$ is $K$-qc.
By  (a) and an application of   Proposition \ref{prop1} to  $\breve{f}$,    (b)  follows.
%definition
%?? The map $z=R(\zeta) = (\zeta- 1/b_0)^{-1} +b_0$   maps    $\overline{\mathbb{C}}$  onto  itself and  $R(\infty)=b_0$, $R(0)=0$.
%Set  $\breve{h}= h\circ R$ and   apply   Proposition \ref{prop1} to  $\breve{h}$.
%If $L=R^{-1}$,  then  $L'(0)=- z_0^{-2}$,  $L(z)=z (b_0  (z-b_0))^{-1}$,  $\arg L= \arg z - \arg T -\arg b_0 $, where   $T=z-b_0$.  ??%

(II) Let   $|z_1|=|z_2|=R<r_0$.  If necessary we can re-numerate points such that  $z_k= R e^{it_k}$, $k=1,2$,  $t_1\leq t_2 \leq t_1 +\pi$
and  $l=l(z_1,z_2)$  be the circular arc defined by  $l(t)= R e^{it}$, $t_1\leq t \leq t_2 $.
%By  $\Delta_l Arg T$ and  $\Delta_l Arg X$  we denote the variation  of $Arg T$ (respectively $Arg X$)  along $l(z_1,z_2)$.
We are going to estimate   the variation  $\Delta_l {\rm  Arg} T$  and  $\Delta_l {\rm  Arg}\, X$.
Since  $\displaystyle{X(z)=X_{p}(z)=\frac{p z}{z- p}}$,
we can write
\begin{itemize}
\item[(i)]   $\arg X= \arg z - \arg T +\arg p $, where   $T=z-p$.\\
Hence
 \item[(ii)]  $\Delta_l {\rm Arg} X \leq \Delta_l {\rm Arg}\, Id + \Delta_l {\rm Arg}\, T $, where $Id$ is the identity map.
\end{itemize}
Since  $T'/T= 1/T$, for      $z=r e^{it}$,
$$ (\arg T)_t= {\rm Im} \left(\frac{T'}{T}i r e^{it}\right).$$
For $z\in B(0,r_0)$,  we  have    $|(\arg T)_t|\leq  \frac{1}{|z-p|} \leq 1/r_0$, and therefore
%the variation  $\Delta_l Arg T $ of $Arg T$ along $l(z_1,z_2)$,
%|\arg T (z_2)-\arg T (z_1)|\leq  |\arg X (z_2)-\arg X (z_1)|\leq
$$\theta(Tz_1,Tz_2)\leq \Delta_l {\rm Arg} T\leq  r_0^{-1}|t_2-t_1|.$$
 Hence, by the  item (ii), for  $|z_1|=|z_2|=R<r_0$,
$$\theta(Xz_1,Xz_2)\leq \Delta_l {\rm Arg} X \leq (1+ r_0^{-1})|t_2-t_1|.$$

%??   Since  $h= \breve{h} \circ  L$,  this  with  an application of    Proposition \ref{prop1} to  $\breve{h}$ yield result.\\

(III)  We only outline a proof.   Set $\zeta=Y(z)$, $\theta=\arg z$,   $\varphi  =\arg \zeta$,  $z=x+iy=re^{i\theta}$,  $\zeta= \xi+i\eta   =\rho e^{i\varphi}$.
So we define the  functions  $\theta=\theta(\rho,\varphi)$, $\varphi=\varphi(r, \theta)$, $\rho= \rho(r, \theta)$  and  $r=r(\rho,\varphi)$.

Since $Y$ is conformal mapping on  $\overline{B(0,r_0)}$  and $Y(0)=0$, by
%?? Kellogg theorem,
the item (h) in subsection  \ref{ssAux},   we find   $\rho(r, \theta)\approx r$  and   $r=r(\rho,\varphi)\approx \rho$.

Let   $\zeta \in \gamma:=\gamma(\varepsilon,c,\mu) $,   $\rho =|\zeta|$  and $z'=X(\rho)=r(\rho,0)e^{i\theta'}$, where  $\theta'=\theta(\rho,0)$.

Case 1.  Suppose that  $p=p_1+ip_2$, $p_2>0$.

Since $X(\infty)=p$, $X$ maps the coordinate axis  $\eta=0$ (in  the $\zeta$ -plane)  onto the circle  $K=C(iR_0,R_0)$   which contains $p$, where  $R_0=R_0(p)=\frac{|p|^2}{2 |p_2|}$ depends only on $p$. $Y_p$  maps  $B=B(iR_0,R_0)$  onto $\mathbb{H}$.
Let $K'$ be semi circle  $y= R_0 - \sqrt{R_0^2-x^2}$. Then  $\theta'\approx x(\rho)\preceq r(\rho)\approx r$.
By the part (II) of the  Proposition, $\theta(z, z')\preceq \varphi\preceq \rho^\mu \preceq r^\mu $. Hence, since   $\theta \leq \theta(z, z') +  \theta'$,  we find
$\theta \preceq r^\mu$ (thus we can choose  $\mu_1=\mu$).

In a similar way we consider:

Case 2.  Suppose that  $p=p_0=p_1+ip_2$, $p_2<0$.  In this case  $Y_p$  maps  $B=B(-iR_0,R_0)$  onto $\mathbb{H}^*$.

Case 3. $p_0 \in\mathbb{R}$.  In this case   $Y_p$  maps  $\mathbb{H}$ onto itself.
\end{proof}
%%%%%%%%%%%%%%%%%%%%%%%%%%%%%%%%%%%%%%%%%%%%%%

\begin{theorem}\label{thmain0}
%(H-0)  {\rm (}$\mathbb{H}^0_{qc}${\rm )}
%Suppose
\begin{enumerate}
\hfill
%\item [\rm{(i)}] Suppose that $D$ satisfies the hypothesis  {\rm (Lyp-0)} \rm{(}see Definition  \ref{1.3}\rm{)} and   $h$ is a $K$-qc mapping of $\mathbb{H}$ onto $D$ and %$h(0)=0$.\\
%$h:\mathbb{H}\rightarrow D$ is $K$-qc map, where $\mathbb{H}$ is
%the upper half-plane and suppose $D$ is a Jordan domain with
%$C^{1,\mu}$ boundary $\partial D$  positively oriented, and  suppose %that $0 \in\partial D$, $h(0)=0$, and that $D$ has real axis as %tangent at $0$, with inner normal pointing upwards.
%Then  by the item (i) from  subsection \ref{ssAux}, we have:

\item [\rm{(I)}]   Suppose  \rm{(i)}:   $h$ is a $K$-qc map from  $\mathbb{H}$ onto a Lyapunov domain $D$. Then    $h$  has a $K_1$-qc extension to a map $\tilde{h}$    of the complex
plane.
%which by abuse of notation we denote again by $h$  if there is no possibility of confusion.\\
%,    $b_0=h^{-1}(\infty)\in \mathbb{C}$.
\item [\rm{(II)}] If   $h$ satisfies  the hypothesis  $\mathbb{H}_{qc}^0$, then there is a constant $l_0= 16\, 2^\alpha k_1^{-1} $ which depends on  $K_1$  and the Kellogg multiplicative constant  of $D$  \rm{(}with respect to $a_0=h(i)$\rm{)}  $k_1=\underline{k}_1(D,a_0)$, such that
\item[\rm{(ii):}]  $|h(z)|\leq l_0 |z|^{1/K_1}$ if  $z\in \mathbb{H}$  and $|z| \leq 1$.
% ?? Suppose that
%In addition, ?? wlg  suppose that $D$ satisfies the hypothesis ?? (d) in subsection \ref{ssAux}.
\item [\rm{(III)}] If  $D$ satisfies the hypothesis  {\rm (Lyp-0)},  then  there are constants $\varepsilon >0$ and $c>0$  such that
for $|w|<\varepsilon$ (here $c\varepsilon^\mu<\pi$)
and $w\in \partial D$, either $|arg(w)|<c|w|^\mu$
or $|\pi-arg(w)|<c|w|^\mu$ where $arg$ is the branch of the argument
determined by  $-\pi/2<arg(w)<3\pi/2$ and moreover that set
$$D_0=D_0(\varepsilon)={\rm Lyp}(\varepsilon,c,\mu)=\{w : c |w|^\mu<arg(w)<\pi-c|w|^\mu\,, |w|<\varepsilon\}$$
satisfies $D_0\subset D$, where  $c$ depends only on the Lyapunov  multiplicative constant of $\partial D$.

\item[\rm{(iii):}]  We can choose  $c=l_2= l^2_D= \frac{\pi}{2}\,l_1\, b_\gamma^{1+\mu} $, where  $l_1=lyp(\gamma)$  and  $b_\gamma=b_\gamma^{arc}$ is the arc-chord constant of  $\gamma$.
%Suppose also that $A, 1/A$ are H\"{o}lder continuity constants for
%$K_1$ for $h$ .
\item [\rm{(IV)}] Then there is a constant $c_1=c_1(\mu,
\varepsilon, c, K_1, l_2, |p|)$ such that the region
$$H_0= H_0 (\varepsilon) =\{z: c_1
|z|^{\mu/K_1^2}<arg(z)<\pi-c_1|z|^{\mu/K_1^2}, |z|<(\varepsilon/l_0)^{K_1}$$
satisfies
\item [(a)]  $h(H_0)\subset D_0$  and
\item [(b)]   there are  constants  $\varepsilon_2,c_2,\mu_2$  such that  $D'_0\subset h(H_0)$, where  $D'_0 ={\rm Lyp}(\varepsilon_2,c_2,\mu_2)$.
\end{enumerate}
Note that $H_0= H_0 (\varepsilon)={\rm Lyp}(\varepsilon_1,c_1,\mu_1)$,    where    $\mu_1=\mu/K_1^2$ and   $\varepsilon_1= (\varepsilon/l_0)^{K_1}$.
\end{theorem}
%%%%%%%%%%%%%%%%%%%%%%%%%%%%%%%%%%%%%%
 Recall that the hypothesis  \rm{(i)} (together with some technical  requirements  $h(0)=0$ and that $D$ satisfies (Sp0))  in the theorem   is essentially  equivalent to the hypothesis  ($\mathbb{H}_{qc}^0$).  From the proof below  it is clear  that the hypothesis  \rm{(i)} implies
\begin{itemize}
  \item[\rm{(i1):}] $h$ is a qc mapping of $\mathbb{H}$ onto the quasidisk  $D$  (which is  much weaker  then  \rm{(i)}),
 and that  the statement   \rm{(I)}     holds   under the hypothesis  \rm{(i1)}.
\end{itemize}
If  in addition to  (i1),  $h(0)=0$  and $0\in \partial D$,
we leave to the interested  reader to  state and prove  a  corresponding version of the statement  \rm{(II).
% Note  if only
Since $h^{-1}$  is also qc  the proof  of  \rm{(IV)}  of  Theorem  \ref{thmain0}  shows  that the following holds
\begin{itemize}
  \item[$\rm{(IV')}$:] for each   special  domain of Lyapunov type $X_0$ with vertex  at  $0$, there is a special  domain of Lyapunov type $Y_0$ with vertex  at  $0$     such that
 $Y_0\subset h(X_0)$.  In particular, we can choose $X_0=H_0 $   and  $Y_0$ to be   an elementary  Lypunov domain  $D_0^- $      such that $D_0^- \subset h(H_0)$.
\end{itemize}
%%%%%%%D_b
%Figure  1.

\begin{figure}[h!]
\begin{center}
\includegraphics[width=\linewidth]{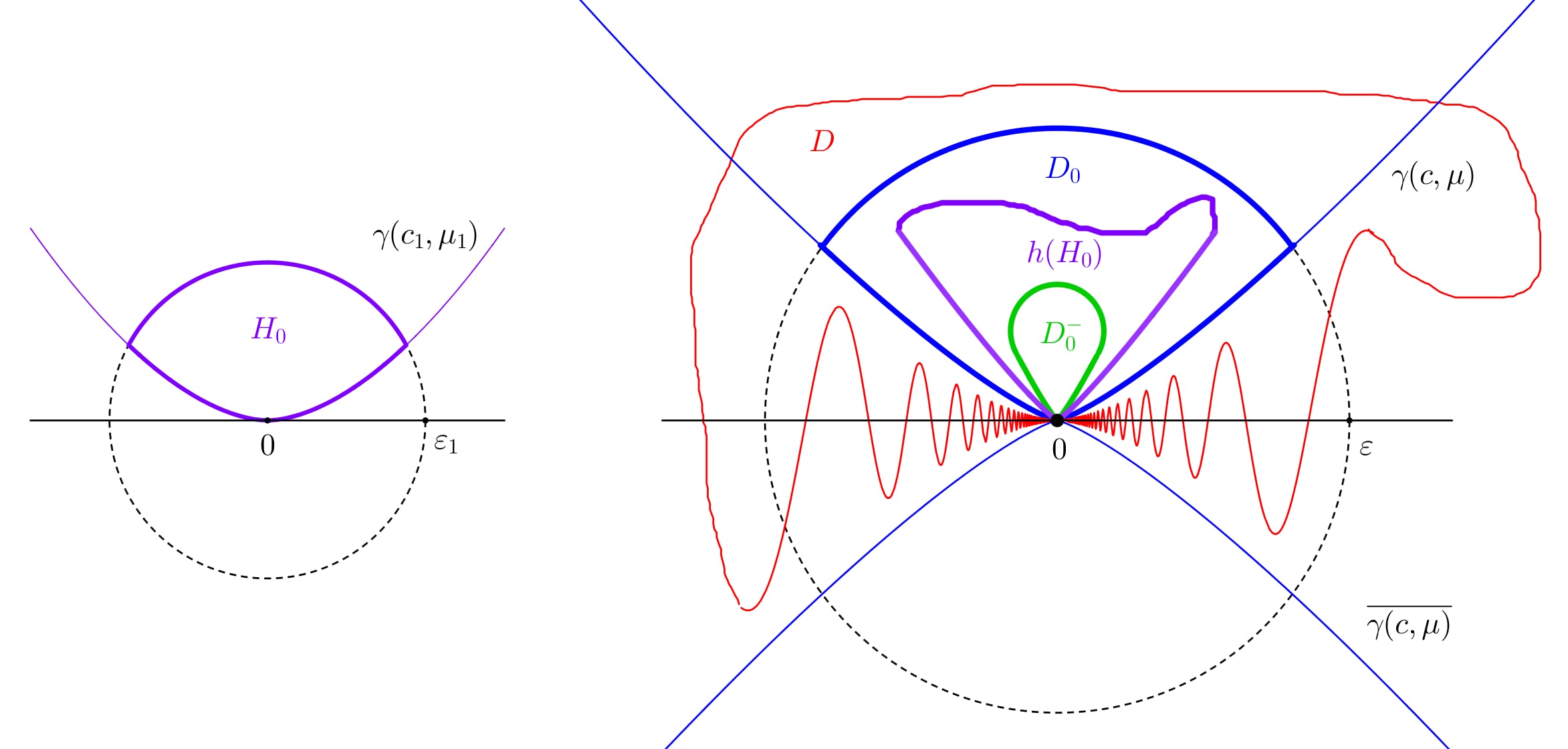}
\caption{}\label{fig:1}
\end{center}
\end{figure}

On the figure \ref{fig:1} the domains $D$, $D_0$, $D_0^-$,  $H_0$  and  $h(H_0)$  are enclosed by  lines   whose  colors  are  red, blue, green,  violet (on the left)  and  violet (on the right) respectively.
%%%%%%%

\begin{proof}
%Concerning the statement (II),  note that
%$l_0>1$ in general and      in the case that the codomain is  the  unit disc $D=\mathbb{U}$ with  $h(0)=0$ we can directly  apply Mori's theorem, $l_0$  is  $16$ .

\rm{Proof of (I)}.    Since  $D$ is  $C^{1,\mu}$,  $D$  is a
quasi-circle  and therefore by the item (i) from  subsection \ref{ssAux},    the statement  $\rm{(I)}$   follows.

\rm{Proof of (II)}.  We suppose that  $a_0=h(i)$ is given. Let  $B_1=B(i;1)$  and  $R_1$ a conformal mapping of $B_1$ onto   $\mathbb{H}$ such that $R_1$ fixes $0$ and $i$   and $R_2$ a conformal mapping of $D$  onto  $B_1$  such that $ R_2(0)=0$  and   $R_2(a_0)=i$.
Set  $\underline{h}= R_2\circ h \circ R_1$. Then  $\underline{h}(0)=0$, $\underline{h}(i)=i$, $\underline{h}$ maps   $B_1$ onto itself,   $h = R_2^{-1} \circ \underline{h} \circ R_1^{-1}$.   Set   $w=h(z), \zeta= R_1^{-1}(z), \zeta'= R_2^{-1}(w)$   and  $\zeta'= \underline{h}(\zeta)$.
Note that   $R_1^{-1}(\mathbb{H})$  is a disk.
By the item  (g) in   subsection  \ref{ssAux},  we find   $|R_1^{-1}(z)|\leq 2 |z|$, $z\in \mathbb{H}$.
Since       $D$ is   a  Lyapunov  domain bounded,
  $|R_2^{-1}(\zeta')|\leq k_3 |\zeta'|$, $\zeta'\in B_1$, where  $k_3= \underline{k}_1(D,a_0)^{-1}$.
It is clear  that  $\underline{h}$   is $K_1$-qc   with $\underline{h}(i)=i$.    Now  an application  of  Mori's theorem to
%$h= R_2\circ h \circ R_1$
$\underline{h}$  on $B_1$,  shows that $\underline{h}$ is $\alpha$ -H\"{o}lder continuous on $B_1$  with a multiplicative   constant  $16$
%$l^1=l^1(K_1)$
, where   $\alpha=1/K_1 $ and in particular   $|\zeta'|= |\underline{h}(\zeta)|\leq 16 |\zeta|^\alpha$.
%$\mathbb{U}$.   Suppose also that $l_0, 1/l_0$ are H\"{o}lder continuity constants for $h$ on $\mathbb{U}$.
Since   $h(0)=0$, there is  a constant  $l_0=l_0(K_1)$  such that  the part (ii)  of the theorem holds:
$|h(z)|\leq l_0 |z|^{1/K_1}$ if $z\in \mathbb{H}$  and  $|z| \leq 1$.

Proof of \rm{(III)}.  Set    $l_1=lyp(\gamma)$.  Since  $D$ satisfies the hypothesis  (Lyp-0), the item (d) in subsection \ref{ssAux}   holds.    By  the item  (d),  there is $\epsilon_1>0 $ such that  the trace   of the path  $\gamma_1$ which is  defined in (d)  by $v=F(u)$, $u\in [-\epsilon_1,\epsilon_1]$, where  $F$ is $C^1$,  is on  $\partial D$.
Let  $\gamma_2=\partial D \setminus tr(\gamma_1)$. Then   there is a constant $\epsilon_2 >0$  such that $\gamma_2$ has no points in the disk $B(\epsilon_2)$. Let  $L$  be  the  length of  $\partial D$ and  $\displaystyle{\hat{\gamma}}$ a parametrization  of the positively oriented boundary  $\partial D$    by the arc-length parameter $s$, where  $s \in[0,L]$  and  $s(0)=0$ (we need  arc-length parameter  only around $0$).   Set  $\displaystyle{\hat{\gamma} (s)= u(s) +i v (s)}$  and $w=R e^{i \Psi}$. Then
$$|v (s)|= |v (s) -v (0)|= |v' (s^1)| s\leq c_1 s^{1+\mu},$$ where  $c_1=l_1$.
% and
Since  a  quasicircle   satisfies  the arc-chord condition, we have    the following:

(a1) If a  $C^1$ curve is a quasicircle, then     $s \leq c_2 R $,    where     $c_2=  b_\gamma=b_\gamma^{arc}$.

Also we can prove  the following version of (a1).\footnote{we observed it after writting  a revision of the manuscript.}

(a2) By definition of  Lyapunov  curve  $\hat{\gamma}'(t)=1 + \epsilon(t)$, where  $|\epsilon(t)|\leq c t^\mu$.  Hence
$\hat{\gamma}(s)=\int_0^s\hat{\gamma}'(t)dt$  and therefore  there is  $\tau=\tau(c,\mu)$  such that  $s\leq 2 |\hat{\gamma}(s)|$ for $s \leq \tau$.

%%%%%%%%%%%%%%%%%%%%%%%%%%
Hence,  there is $\epsilon>0$  such  that
\begin{itemize}
 \item[(iv)]    $|arg \hat{\gamma} (s)|< \frac{\pi}{2}$ for   $s \in (0,\epsilon]$   and
\item[(v)]   $\frac{\pi}{2} < arg\,\hat{\gamma} (s) < 3\frac{\pi}{2}$ for   $s \in [L-\epsilon,L)$.
\end{itemize}
We can choose   $\epsilon <  \epsilon_2$.
Hence, for $s$ small ($0<s<\epsilon$),  we find
$$R |\Psi|\leq \frac{2}{\pi} R |\sin(\Psi)|=\frac{2}{\pi}v (s) \leq c_3 s^{1+\mu},\quad    \mbox{where}\quad   c_3= \frac{2}{\pi}c_1,$$  and therefore
there is a constant $c$  such that
\begin{itemize}
\item[(vi)] $|\Psi|\leq c R^{\mu}$,   where $c$  is given by the item (iii) in   Theorem  \ref{thmain0}.
\end{itemize}
Using  the mapping  $A(w)=-w$  and   (vi), we find that
\begin{itemize}
\item[(vi')] $|\pi - \Psi|\leq c R^{\mu}$  for   $s \in [L-\epsilon,L)$.
\end{itemize}
From  (vi)   and  (vi'),     (III)  follows.

Proof of \rm{(IV)}. We use the notation from  Proposition \ref{prop0}.  Set  $\breve{h}=h \circ Y_p$,  where    $p=h^{-1}(\infty)\in \mathbb{C}$. By the statement (I)  and  easy part of  Proposition \ref{prop0},  $h=\breve{h}\circ Y_p$, and $h$   is $K_1$-qc mapping of the plane $\overline{\mathbb{C}}$  onto  itself,  with   $\breve{h}(0)=0$ and      $\breve{h}(\infty)=\infty$.
We use
notation: polar coordinates
%$z= r e^{i\varphi}$ in $z$-plane
$\zeta= \rho e^{i\varphi}$ in the  $\zeta$-plane
and $w=R e^{i \Psi}$ in the $w$-plane.
%More generally, suppose that $D$ is quasi circle and that a curve $C_1$: $\Psi=\Psi_1(R)$,  $0\leqR\leq R_0$,  in $D$.
%Hence there is the  part $C_0$ of the  boundary of $D$  around $0$
% which      is the above of $C_1$.   Suppose further  that $\Psi_1$ is %increasing  in $R$. %and $|\Psi_0|  \leq \kappa_1 \Psi_1$.
Recall  by \rm{(III)} there  is   a curve  $\gamma=\gamma(c,\mu, R_0)$   in $D$. Set    $\gamma_0$: $\Psi=\Psi_0(R)= c
R^\mu$,  $0\leq R\leq R_0$. Hence there is
a part $C_1$ of the  boundary of $D$  around $0$ (say the right  half part)
 %is given by  $ \Psi=\Psi_0(R)$,  $0\leq R\leq R_0$. Consider ( $C_0$
 which      is below  $\gamma_0$  and which defines  the curve   $\gamma_1$.

Case 1.
%Suppose  first that  $h(\infty)=\infty$.
We first prove  for  $\breve{h}$.
Let $w=R e^{i \Psi}\in \gamma_0$  and let  $w'=R e^{i \Psi'}$   be the intersection  of the circle $T_R$   with  $\gamma_1$.    Then
$\theta(w,w')\leq  (|\Psi| + |\Psi'|) \leq 2 |\Psi|$.
Set    $\zeta=\breve{h}^{-1}(w)$  and $\zeta'=\breve{h}^{-1}(w')$.  Since  $\gamma_0$ is the right  half of  $\gamma=\gamma(c,\mu)$,  $\zeta'>0$  and
$\varphi= \theta(z,z')$.
Hence   using  the quasihyperbolic metric $k$ on
$\mathbb{C}^*$ (Proposition  \ref{prop1}), we have

$\varphi\leq \kappa_2 \Psi^{\alpha}$,  where  $\alpha=1/K_1$.
Since, by (II),   $R\leq c \rho^\alpha$, we find  $\varphi\leq \kappa_2
\big(\Psi_0(c\rho^\alpha)\big)^{\alpha}$ and therefore
%If we choose   $\Psi_0(R)= cR^\mu$,
we get  $\varphi\leq \kappa_3 r^{\alpha^2 \mu}$. Thus   we find:
\begin{itemize}
\item[(vii)]  the
curve  $\breve{h}^{-1}(\gamma_0)$  is below the curve $\gamma(\kappa_3,\alpha^2 \mu)$.
\end{itemize}
% Proof that $h$
%%%%%%%%%
Note that if a curve satisfies $\varphi \leq c \rho^{\mu}$, then it
is  is below the curve $\gamma(c,\mu)$.
Recall that   we set  $\mu_1=\mu/K_1^2$ and   $\varepsilon_1= (\varepsilon/l_0)^{K_1}$.  Note that   $\gamma_0$ is the right  half of  $\gamma=\gamma(c,\mu)$ and that  in a similar way as above we conclude  that
\begin{itemize}
\item[(viii)]     $\breve{h} ^{-1}(\gamma(c,\mu))$ is below the curve
$\gamma(c_1,\mu_1)$, $\rho<\varepsilon_1$.
\end{itemize}
%?? We can choose  $\varepsilon_1$  such that
By the part (ii)  of the theorem,
$h(B(\varepsilon_1)) \subset B(\varepsilon)$ and  it is readable that it  yields (a). Since $\breve{h}^{-1}$  is also qc (a)  implies (b). Thus we have proved
\rm{(IV)} for $\breve{h}$   with $c_1=\kappa_3$.

Case 2.  Proof for $h$.
%Suppose now  that  $b_0=h^{-1}(\infty)\in \mathbb{C}$.
By Case 1,  there is $H_0'$   such   that  $\breve{h}(H_0')\subset D_0$.
By  Proposition \ref{prop0}   there is $H_0$   such   that  $Y(H_0) \subset H_0'$ and it completes proof.
Thus we have proved (a).

Let us prove  that  (a) implies (b).
Namely,  since $h^{-1}$  is also qc, by (a)  there is  $D'_0$  such that
$h^{-1}(D'_0)\subset H_0$  and therefore   $D'_0 \subset h(H_0)$.
\end{proof}

\subsection{Global approximation}\label{ss3.1}
Concerning the previous theorem,  note that  $\mu_1\leq  \mu$  and    $\varepsilon _1\leq  \varepsilon$,
and  in particular,  one can derive (see (IV')):
\begin{itemize}
\item[(a)]
there is  $\underline{\epsilon}=\underline{\epsilon}(\varepsilon,c)< \varepsilon_1$    such that      $h(L^1)\subset L'\subset D_0$, where
 $L'= {\rm Lyp}^- (\varepsilon,c)$ is  $\mu$-Lyapunov   and   $L^1= {\rm Lyp}^- (\underline{\epsilon},c)$ is  $\mu_1$-Lyapunov  and   $L^1 \subset L'$. Hence,  since $h^{-1}$  is also qc,
$h^{-1}(L^1)\subset L'$  and therefore   $L^1 \subset h(L')$.
%\subset D_0$.
\item[(b)]  In a similar way,   there is  $\underline{\epsilon}^1\leq \underline{\epsilon}$   and $\mu_2\leq  \mu_1$   such that  $h^{-1}(L^1_-)\subset L^1$, where
$L^1_- = {\rm Lyp}^- (\underline{\epsilon}^1,c)$ is  $\mu_2$-Lyapunov.
\end{itemize}
Hence we  derive:
\begin{itemize}
\item[$\rm{(IVa)}$]  If  $h$ satisfies  the hypothesis  $\mathbb{H}_{qc}^0$, then     $L^1_- \subset h(L^1) \subset D_0$.
\end{itemize}
Note that  it is easy to transfer  Theorem \ref{thmain0}   to  the setting  of  the unit disk.  Now we show that   the corresponding version of  it  holds with $\mathbb{U}$  instead of  $\mathbb{H}$.

% $\mu$-  \subset h(L')
We first need a version of $\rm{(IV')}$  for $\mathbb{U}$   with   special  Lyapunov   convex  domains.
Note that  $H_0$  has two cusps.
%Let  $B^0$ be the maximal disk in  $H_0$ with center at $i\varepsilon_1$.
%There is  a  Lyapunov -  $ C^{1,\mu_1}$ domain $H_0^-= H_0^-(\varepsilon)$  which touches $0$ and which is a subset of $H_0$ and contains  $B^0$.
%%%%%%%%%%%
%The
In this subsection  by $D_0$  we denote the set defined in  Theorem \ref{thmain0}.
\begin{definition}[${\rm lyp}(D)_b$]\label{dlyp2}
\begin{itemize} Here we define   $\underline{A}_0$,  $R_a$, $T_b$ and   $h_a$.
\item[(i)] Consider the conformal
mapping $A_0=\underline{A}_0$ defined  by  $\underline{A}_0(z)= \frac{4i-z}{4i+z}$;   $\underline{A}_0$  maps   $\mathbb{H}$ onto $\mathbb{U}$ such that $\underline{A}_0(0)=1$  and $\underline{A}_0(-4i)=\infty$.

%:    Let  $h:\mathbb{U}\rightarrow D$ be harmonic  $K$-qc map, where
%$\mathbb{U}$ is the unit disk  and suppose $D$ is a Jordan domain
%with $C^{1,\mu}$ boundary $\partial D$ positively oriented, and suppose that $0 \in \partial D$, $h(1)=0$, and that $D$ has real axis as tangent at  $0$, with inner normal pointing upwards.\\
\item[(ii)]  For   $a=e^{i\alpha}  \in \mathbb{T}$   define
$R_a(z)=e^{i\alpha}z$,  and for  $b\in \partial D$   if    the unit inner normal $n_b=e^{i\beta}$ at $b$ exists,  we define   $T_b(w)=\underline{T}_b(w)= - i e^{i\beta} w +b$, and
\item[(iii)] if $D$ satisfies the hypothesis  {\rm (Lyp-0)}, we define    $\hat{D}_b= T_b (\hat{D}_0)$, where $\hat{D}_0$  is defined in the item  (A) below.  If we wish to indicate that $\hat{D}_b$ is an elementary Lyapunov domain we use notation    ${\rm lyp}(D)_b$.
%It is clear that  $D_b^-= T_b (D_0^-)$
\item[(iv)] For $a\in \mathbb{T}$,  set  $h_a=h_a^b :=\underline{T}_{b}^{-1} \circ   h\circ R_{a}$, $a \in \mathbb{T}$,   where $b=h(a)$, and let $\hat{h}=h\circ A_0$ and $\hat{h}_a=h_a\circ A_0$ .
 \end{itemize}
\end{definition}
%\item
\begin{itemize}
\item[(v)]  Next  suppose that $D$ satisfies the hypothesis  {\rm (Lyp-0)} \rm{(}see Definition  \ref{1.3}\rm{)}.

We will prove  that(see also Proposition  \ref{luap1}):

\item[(A)]
there is  $\varepsilon_2>0$  such that  $\hat{D}_0\subset T_{b}^{-1}(D)$   for every $b\in \partial D$, where  $\hat{D}_0= {\rm Lyp}^- (\epsilon_2,c)$.
\hfill
 \item[(vi)] In addition to  (v)  suppose that   $h$ is a qc mapping of $\mathbb{U}$ onto $D$ and  $h(1)=0$ (that is  $h$ satisfies  the hypothesis  $\mathbb{U}_{qc}^0$).
\end{itemize}
Then  $\hat{h}$ is a qc mapping of $\mathbb{H}$ onto $D$ with   $\hat{h}(0)=0$.

It seems useful to consider the following properties, which is an immediate corollary  of $\rm{(IVa)}$:
% will be improved  in  Proposition  \ref{luap1}). ??  Hence,  by  ,
\begin{itemize}
\item[(B)]
%\hat{H}_0  in $\mathbb{H}$
If   $h$ satisfies  the hypothesis  $\mathbb{U}_{qc}^0$,  there are    corresponding    elementary  Lyapunov  domains $H^1\subset \mathbb{H}$ and   $D^1_-\subset D$   with vertex at $0$    such that
$D^1_{-}\, \subset \hat{h}(H^1) \subset \hat{D}_0$. See Proposition  \ref{luap1} for a stronger result.
\item[(vii)]  Further set    $U_1= A_0(H^1)$  and   $U^1_- =A_0(D^1_-)$   and set  $U_a=R_a (U_1)$ .
\end{itemize}
Note that  $H^1$   is a  special   $\mu_1$-Lyapunov. Thus we have
\begin{itemize}
\item[(C)]     $U^1_-  \subset h(U_1)\subset\hat{D}_0$.
\end{itemize}
%%%%%%% more general
In order to state a corresponding   form of   $\rm{(IVa)}$  for $\mathbb{U}$,  it is convenient to call $V=A_0(L)$  an elementary domain if $L$ is elementary (see also Proposition \ref{prop0}).
Now,  by $\rm{(IVa)}$, it is clear  that we have:
\begin{itemize}
\item[$\rm{(IVb)}$]If  $h$ satisfies  the hypothesis  $\mathbb{U}_{qc}^0$  and  $L$ is an elementary  Lyapunov  domain with vertex at $0$  in $\mathbb{H}$,
then  there are  elementary Lyapunov  domains   $V_1$ and   $V^1_-$ in $\mathbb{U}$  with vertex at $1$   such that    $V^1_-  \subset h(V_1)\subset L$.
\end{itemize}
%%%%%%%%% ??  $\mu$-

  (D)   Now,  we also  suppose that  $h$ satisfies  the hypothesis  $\mathbb{U}_{qc}^0$.
Recall by the item (i) from  subsection \ref{ssAux},  then  $h$  has a $K_1$-qc extension to a map $\tilde{h}$    of the complex
plane.
%,  which by abuse of notation we denote again by $h$  if there is no possibility of confusion.

We can choose  $p=\tilde{p}$  such that  $h(\tilde{p})=\infty$ and  $|\tilde{p}|\geq 3$.
Set   $\hat{p}=A_0^{-1}(\tilde{p})$,  $\tilde{p}_\alpha=h_a^{-1}(\infty)$  and  $\hat{p}_\alpha=A_0^{-1}(\tilde{p}_\alpha)$. Check that $\tilde{p}\notin h(B(0,2))$  and   therefore  $|\hat{p}|\geq 2$.Hence, since   $\tilde{p}_\alpha=h_a^{-1}(\infty)= e^{-i\alpha}\tilde{p}$, we find

$(D0)$:  $h_a^{-1}(\infty)\notin h(B(0,2))$  and   therefore,
$|\hat{p}_\alpha|\geq 2$  for  every   $a=e^{i\alpha}  \in \mathbb{T}$.

Note that   $T_{b} \circ h_a  =h\circ R_{a}$  and that
$h$ satisfies  the hypothesis $\mathbb{U}_{qc}^0$,  ($U$-1) if and only if  $\hat{h}$  satisfies  $\mathbb{H}_{qc}^0$, ($H$-0) respectively.

%By the notation and  statement of Theorem  \ref{thmain0}, we conclude that  \\
%$\rm{(IV'')}$  There is a  special   $\mu$-Lyapunov   convex  domain  $\hat{L}_0 \subset D_0$
%$\mathcal{D}_1$  and   $D_0$ are
%such that
%%%%%%%%%% By (D)  there is  $r>0$  such that  for  each   $w \in \partial D \cap B(r)$  graph with respect   to $uv$-coordinates.
%%%%%%%%%%%%%%%%%%%%%
In addition, we need a property of  $C^1$ domains.   Suppose that domain $D$  is  $uv$-plane.
\begin{lemma}\label{l_app1}
\begin{itemize}\hfill
\item[(i)]
Suppose that a $C^1$ domain  $D$ satisfies the hypothesis \rm{(Sp0)}.
\item[(I)] Then   there is  $r>0$  such that  for  each   $w \in \partial D$,    $\underline{T}^{-1}_w(\partial D) \cap B(0, r)$  is a   graph with respect   to $uv$-coordinates.
\end{itemize}
\end{lemma}
%Proof of  (D).
\begin{proof}
Let  $L$  be  the  length of  $\partial D$ and  $\displaystyle{\hat{\gamma}}:[0,L]\rightarrow \partial D$  a parametrization  of the positively oriented boundary  $\partial D$    by the arc-length parameter $s$.  We also  write  $w= \hat{\gamma}(s)$,   where  $s \in[0,L]$  and  $s(0)=0$.
% By the hypothesis  then   there is  $r_1>0$  such that
%\begin{itemize}
%\item[(ii)]    $|\arg \hat{\gamma}'(s)| < \pi/8 $  or  $|\arg \hat{\gamma}'(s)-\pi | < \pi/8 $  for    $w= \gamma(s) \in \partial D \cap B(r_1)$.
%\end{itemize}
Here  there is  the function  $s=s(w)$  which  is the inverse  of   the function  $w= \hat{\gamma}(s)$ and which maps  $\partial D$  onto $[0,L]$.
Since   $\displaystyle{\hat{\gamma}'}$   is  continuous on $[0,L]$, it is  uniformly  continuous on $[0,L]$.
The function  $s=s(w)$   is continuous  on  $\partial D$   and  hence  $C(w)= \displaystyle{\hat{\gamma}'(s(w) )}$  is continuous  on  $\partial D$  and
uniformly  continuous on  $\partial D$.
Therefore there is $r_2>0$  such that
%Set  $r=r_2/2$.   Then we have:  T_{w_0}
\begin{itemize}
\item[(ii)]   $|\arg \hat{\gamma}'(s_2)-\arg \hat{\gamma}'(s_1)| < \pi/8 $\,   for\,  $|w_2-w_1|\leq r_2$,   where  $w_k= \hat{\gamma}(s_k)$, $k=1,2$.
\end{itemize}
Let us to prove (I)  for  $r=r_2/2$.   Contrary,  suppose that (I) is not true. Then  for some $w_0\in \partial D$,
$D(w_0,r):=(\partial D) \cap B(w_0, r)$  is not a graph  with respect to coordinates determined by unit vectors  $\hat{\gamma}'(w)$  and $n_w$.  Hence   there are two points $w_1$ and  $w_2$  in this set  such that  $w_1 w_2$ is parallel  to  the normal $n_{w_0}$  of $\partial D$ at
$w_0$. Therefore there is $w_3$   in this set  such that     $\gamma'(s)$   at  $w_3$   is parallel  to  the normal $n_{w_0}$. This contradicts     (ii).
Thus
%for $r=r_3$
we have (I).
\end{proof}
%  Further  chose a fixed positive real number  $x_0 \in U_1$.
Using  the approach in  the proof  of  statement \rm{(III)},   \rm{(IV)},     $\rm{(IV')}$ of  Theorem  \ref{thmain0},     $\rm{(IVb)}$  and  (iii), we can  prove
%$\rm{(V)}$:
\begin{proposition}\label{luap1}
\begin{itemize}
\hfill \item [\rm{(a):}]  Suppose that   $D$ satisfies the hypothesis  {\rm (Lyp-0)}.
  %?? One   can choose $\varepsilon$ in part  \rm{(III)} (that is  $\hat{D}_0=\hat{D}_0(\varepsilon)$ and  $U_1=U_1(\varepsilon)$) such that  \\
\item [\rm{(I):}]  there   is   an   elementary    Lyapunov domains   $\hat{D}_0$ in $D$   with vertex at $0$   such that
% and   $L^1$
$\hat{D}_0\subset\underline{ T}_{b}^{-1}(D)$.
%that is  ${\rm lyp}(D)_b\subset D$  for every $b\in \partial D$,    and
\item [\rm{(b):}] In addition to  {\rm(a)}   suppose that   $h:\mathbb{U}\rightarrow D$ is a qc homeomorphism.
\item [\rm{(IIa):}] Then  there is  a  Lyapunov domain  $\hat{U}_1$  in $\mathbb{U}$  with vertex at $1$  such that    for every  $a \in \mathbb{T}$,  $h_a(\hat{U}_1)\subset \hat{D}_0$.
%, that is   $h(U_a)\subset {\rm lyp}(D)_b$, where $U_1=A_0(L^1)$    and  $b=h(a)$.\\
\item [\rm{(IIb):}]  In addition, there   is  an  elementary    Lyapunov domain   $D_0^-={\rm Lyp}(\varepsilon_0^-,c_0,\mu)$ in $D$ with vertex at $0$    such that
$D_0^-\subset h_a(\hat{U}_1)$ for every  $a \in \mathbb{T}$.

\item [\rm{(III):}]   $D_0^-\subset h_a(\hat{U}_1)\subset   \hat{D}_0$
\end{itemize}
\end{proposition}
Note that in general $h_a(\mathbb{U})$  is not a fixed domain  for  $a\in \mathbb{T}$ and therefore we need  first to  consider the part (I) and then the part (II).
\begin{proof}

%This is the statment (I6), proved in subsection  \ref{ss}.
% =\underline{F}(h  \mathbb{U} = \underline{D}

(I).    $D_b$ satisfies the hypothesis  {\rm (Lyp-0)} for every $b\in \partial D$.  Consider the   family $\underline{\mathcal{D}}:=\{D_b=\underline{T}_b^{-1}(D) : b\in \partial D\}$.  For $b\in \partial D$  define   $\varepsilon(b)$ to be maximum of $\varepsilon$ for which
${\rm Lyp}(\varepsilon,l_2(D))\subset \overline{D_b}$, where $l_2(D)$  is  the second Lyp-constant.
By Lemma  \ref{l_app1},   there is  $r>0$  such that    $D_b \cap B(r)$ is  a  graph with respect   to $uv$-coordinates   for  each $b\in \partial D$.
Since    all domains $D_b$, $b\in \partial D$,  have the same   Lyapunov  multiplicative constants,
using (iii) and  the  approach in the  proof  \rm{(III)}  of  Theorem  \ref{thmain0}, we can prove  that there is $\varepsilon_0 >0$   such that $\varepsilon(b)\geq \varepsilon_0$ for  all  $b\in \partial D$,   and therefore   an  elementary   Lyapunov  domain $\hat{D}_0$
such that $\hat{D}_0\subset D_b$  for every  $ b\in \partial D$  and (I) follows.\\
In addition,  it seems that  we can prove  that the function  $\varepsilon(b)=\varepsilon_D(b)$  is continuous  with respect to $b$.
%D_b=T_b^{-1}(\mathbb{U}  a  special   We suppose that  $a_0=h(i)$ is given

(IIa):    Recall that we use the  notation     $a_0=h(0)$  and $\hat{h}_a=  h_a\circ \underline{A}_0$.
%There  is $r_0> 0$   such that   $B(a_0;r_0)\subset D$.
%$B(T_{b}^{-1}(a_0);r_0)\subset D$.
Let  $\omega_b$  be a conformal mapping of $D_b$ onto $\mathbb{U}$  such that   $\omega_b(\underline{T}_{b}^{-1}(a_0))=0$  and $\omega^b= \omega_b\circ \underline{T}_{b}^{-1}$. Since $\omega^b(0)=0$, by the Kellogg-Warshawski  theorem there are two positive constants  $l_1$ and $l_2$  such that
   $l_1\leq |\omega'(w)|\leq l_2  $, $w \in D$.    Since  $\underline{T}_{b}^{-1}$ is a euclidean isometry,  we have $|(\underline{T}_{b}^{-1})'|=1$
on $D$,  and  therefore
   $l_1\leq |\omega_b'(w)|\leq l_2  $, $w \in D_b$.
Hence,   since   $\omega_b\circ h_a$  maps $\mathbb{U}$  onto itself,     $\omega_b\circ h_a(0)=0$ and $h_a(1)=0$
there is a constant $l^0$ which depends on  $K$  and the Kellogg multiplicative constant  of $D$(with respect to $a_0$), such that for all $a \in \mathbb{T}$,
  $|h_a(z)|\leq l^0 |z-1|^{1/K}$ if  $|z| \leq 1$.

Using $\underline{A}_0$  we can get the corresponding result  for  $\hat{h}_a$:
there is a constant $l_0$ which depends on  $K_1$  and the Kellogg multiplicative constant  of $D$, such that for all $a \in \mathbb{T}$,
\begin{itemize}
\item [(iv):]   $|\hat{h}_a(z)|\leq l_0 |z|^{1/K_1}$ if  $|z| \leq 1, {\rm Im }z\geq 0$.
\end{itemize}
 The functions $\hat{h}_a$, $a \in \mathbb{T}$, are  $K_1$-qc.   By    (iv),  an  application  of  $\rm{(IVb)}$  to the functions $\hat{h}_a$, $a \in \mathbb{T}$,
and the statement  ($D0$) (from   subsection   \ref{ss3.1}),   and  the item (iii) of  Proposition \ref{prop0} with $r_0=1$ to the functions $\hat{h}_a$, $a \in \mathbb{T}$, show that   there is  a  Lyapunov domain  $\hat{H}_0$  in $\mathbb{H}$  with vertex at $0$  such that   $\hat{h}_a(\hat{H}_0)\subset \hat{D}_0$. Set   $\hat{U}_1=A_0(\hat{H}_0)$.
It yields  the  proof of (II).
%%%%% X1
\begin{itemize}
\item [(IIb):]    Using  $\rm{(IVb)}$, since the corresponding parameters are the same for $h$ and $h_a$, one can get (III).
\end{itemize}
\item [(III)]  It is clear  that  (IIa)  and  (IIb) can be stated as (III).
\end{proof}

%Using  statement \rm{(IV)}  of  Theorem  \ref{thmain0}, weprove\rm{(V)}.
%(see  V)
It is convenient to introduce the following notation:
\begin{itemize}
\item [(v):] For  $b\in \partial D$,  set  $D_b^-={\rm lyp}(D)_b^-= \underline{T}_b (D_0^-)$  and  for $a\in \mathbb{T}$, $\hat{U}_a= R_a (\hat{U}_1)$.
 If we wish to indicate that $D_b^-$ is an elementary Lyapunov domain we use notation    ${\rm lyp}(D)_b^-$.
\end{itemize}
Now using euclidean isometry  $\underline{T}_b$ it is easy to get the corresponding results  of the property (III) of  Proposition \ref{luap1}  for domains with vertexes at $b$.
Namely,  by  the property (III) of  Proposition \ref{luap1}   we have     $\underline{T}_b(D_0^-)\subset \underline{T}_b(h_a(\hat{U}_1))\subset  \underline{ T}_b(\hat{D}_0)$.  By the definitions
$\underline{T}_b \circ h_a (\hat{U}_1)= h \circ R_a (\hat{U}_1)=h(\hat{U}_a)$
and therefore the part (I) of the next theorem follows.
%with
%Lemma \ref{le_Lyp1}
By (L1) (see the introduction) we get  the part (II).  So we have
%yields
the crucial result:
%Theorem \ref{tmain1}
\begin{theorem}\label{tmain1}
Suppose that $D$  is  a  Lyapunov   domain  and  $h:\mathbb{U}\rightarrow D$ is a qc homeomorphism. Then
\begin{itemize}
\item[\rm{(I)}]  For every  $a \in \mathbb{T}$,   ${\rm lyp}(D)_b^- \subset h(\hat{U}_a)\subset {\rm lyp} (D)_b$, where $b=h(a)$.
\item[\rm{(II)}]  If  $w-b$ is in the direction of the normal vector $n_b$ then,
%Lemma \ref{le_Lyp1},
$d_b(w)\approx |w-b|$  if $ \varepsilon_2=\varepsilon_2(c_2,\mu_2)$ is a small enough constant.
\end{itemize}
\end{theorem}

\section{Proof that $h$ is  co-lipschitz}\label{s_co_lip}
%, Version 2XX  Here we give more rigorous  proof than  in the  previous subsection.  %prove:
Here we give   proof of co-Lip   property:
\begin{theorem}\label{thmain2}
Suppose $h:\mathbb{U}\rightarrow D$ is a hqc homeomorphism, where
 $D$ is a  Lyapunov domain with $C^{1,\mu}$ boundary, i.e. belonging to
$\mathcal{D}_1$. Then $h$ is co-Lipschitz.
\end{theorem}
We first   need a few results mentioned in the introduction.
\begin{theorem}[Theorem 1.3,  \cite{napoc1}]\label{thm.space1}
\begin{itemize}
\hfill \item[\rm{(i)}]  Suppose  that   $\,h\,$ is a euclidean harmonic complex valued  mapping
from the unit ball $\mathbb{B}\subset \mathbb{R}^n$ onto  a
bounded   domain  $D=h(\mathbb{B})$, which
contains the  ball  $\,B(h(0);R_0)\,$  and there is a half space $H_b$  which touches the  point  $b\in  \partial D$  such that  $D=h(\mathbb{B})\subset H_b$.  Then
\item[$\rm{(I)}$]  $d(h(z),b) \geq (1-|z|)\overline{c}_n R_0 $, \,
$z\in \mathbb{B}$,  where $\overline{c}_n=
\frac{1}{2^{n-1}}$.
\end{itemize}
\end{theorem}
%%%5
Sometimes,  we refer  to  this  result  as a version of  Harnack's   lemma.

In  \cite{napoc1}  we stated  this result under the  condition that the domain  $D=h(\mathbb{B})$ is convex. But, a slight modification
of the proof of  Theorem 1.1  \cite{napoc1} (planar case) shows that  the theorem holds under  the hypothesis (a).
% every

{\it Proof of} $(I)$. We only outline an argument.  To  $b\in \partial D$  we associate
a nonnegative harmonic function $u=u_b$.
%Since  $D$ is convex, for$b\in\partial D$
% there where
Let  $\Lambda_b$ be  the boundary of  $H_b$  and  let  $n=n_b\in T_b\mathbb{R} ^n$  be   a  unit vector  such that  $\Lambda_b$  is  defined by $(w-b,n_b)=0$.
By hypothesis,  $\Lambda_b$  is a  supporting hyper-plane  such that
$(w-b,n_b)\geq 0$ for every $w\in \overline{D}$.
%let     $\Lambda_a$ be  the line  defined by $(w-a,n_a)=0$, where
%$n=n_a$  is    is unit inner normal at $a$ with respect to  $\partial D$.
% a unimodular complex number such
% Since  $D$ is convex,   then  $(w-a,n_a)\geq 0$ for every $w\in \overline{D}$.
Define $u(z)=(h(z)-b,n_b)$  and  $d_b= d(h(0),\Lambda_a)$.  Then
$u(0)=(h(0)-b,n_b)= d(h(0),\Lambda_a)$.  Let  $b_0 \in \Lambda_a$
be the point such that  $d_b= |h(0)- b_0|$.  Then   from the geometry it is clear that $d_a \geq R_0$,etc (one can follow the proof from \cite{napoc1}).
\hfill  $\Box$
%\Omega=
\begin{proposition}\label{Harprop1}
Suppose  that   $\,h\,$ is    a  euclidean harmonic mapping
from the Lyapunov domain  $G$  into  a domain $\Omega$   and \\
(i)  there is a half space $H_b$  which touches a  point  $b\in  \partial \Omega$  such that    $h(G)\subset H_b$.\\
%the curve  \mathbb{B}
Then  $d(h(z),b) \succeq d_G(z) $,    $z \in G$.
\end{proposition}
Note if   $h: \mathbb{U} \rightarrow D$ satisfies  hypothesis  $U$-1,  in general  a point $b\in  \partial D$  does not satisfy the
hypothesis (i). We use  elementary Lyapunov domain  described  in  Proposition  \ref{luap1} to apply this proposition.
\begin{proof}
Let   $\phi: \mathbb{U}\rightarrow G$ be a conformal mapping   and   $h_1=h\circ \phi$.  Application  of Koebe's theorem to  $\phi$   and
Theorem   \ref{thm.space1}  on  $h_1:\mathbb{U}\rightarrow G$ yield the result.
\end{proof}
%%%%%% first
Now  we  illustrate relation between  the circles and special Lyapunov curves and then   prove Lemma \ref{le_Lyp1}.

If   $M(0,d)$, $d>0$,  then the circle $C$  with  center  at $M$ and radius  $d$ is
given by the equation $x^2 +(y-d)^2=d^2$  and the  half -circle  $C^-$  with  $y=d- (d^2 -x^2)^{1/2}$.
Hence    $d-y=(d^2 -x^2)^{1/2}=d (1- x^2/d^2)^{1/2} = d(1- x^2/2 d^2 + o(x^2)= x^2/2 d   + o(x^2)$  and therefore  $y= \frac{1}{2 d}x^2 + o(x^2) $.
 The graph of the curve  $\gamma(c,\mu; \epsilon)$, where  $c= 1/ d , \epsilon=d$ is above the   half -circle $C$.
%lemma \ref{le_Lyp1}  $\gamma(c,\mu; \epsilon)$
%the curve $f(c,\mu)$ For  $c>0$, $0<\mu < 1$,  and $x_0 > 0$ the curve $f(c,\mu)$  in the xy-plane is  defined  by ,  $|x|< x_0$,
\begin{lemma}\label{le_Lyp1}
 For  $c>0$, $0<\mu < 1$,  and $x_0 > 0$, let the curve $C$ be defined by  \rm{(} the curve  $C$ is  defined in Example  \ref{ex1} and denoted  by  $f(c,\mu)$ \rm{)}
\begin{itemize}
\item[(1):]  $y=C(x)=cx^{1+\mu}$,$|x|< x_0$,    and
\item[(2):] let  $M(0,d)$, $d>0$,   be a point and $d'$ the distance from $M$  to the graph  of the curve \rm{(1)}.
\item[(I)]  Then  $d' \succeq d$   and
\item[(a)]  there is  an $\epsilon^0>0$  such that if   $d\leq C(\epsilon^0)$,   then  $d\leq  2  d'$.
\item[(b)]  For $ \epsilon^0 $ we can choose the positive solution of the equation   $c^2(1+\mu)x_1^{2\mu}= 1$.
\end{itemize}
\end{lemma}
\begin{proof}   Let    $d'= |M-M'|$,   where  $M'(x_1,y_1)$.
%$M'(x_1,y_1)$
Since
$y'(x)= c (1+\mu) x^{\mu}$,  we find   that     $k= \frac{d-y_1}{x_1}=   (c (1+\mu) x_1^{\mu})^{-1}$.
%,  andthen     $k= \frac{d-y_1}{x_1}=(c (1+\mu) x_1^{\mu})^{-1}$.
Hence,  $d-y_1=  \frac{1}{c(1+\mu)}  x_1^{1-\mu}$  and  $d=  \frac{1}{c(1+\mu)}  x_1^{1-\mu} +y_1= \frac{1}{c(1+\mu)}  x_1^{1-\mu}(1 +c^2(1+\mu)x_1^{2\mu})$.
Hence,   $d\leq  2  d'$ if   $ (1 +c^2(1+\mu)x_1^{2\mu})\leq 2 $ (that is  if $ \varepsilon_2=\varepsilon_2(c_2,\mu_2)$ is a small enough constant).
\end{proof}
%(see subsection \ref{ss})
 We are now ready to finish the proof.  We will apply  Proposition  \ref{luap1} and notation used there,  and  Theorem  \ref{thm.space1}.
Further  chose a fixed positive real number  $x_0 \in \hat{U}_1$.
%%%%%%%%%%%%%%%%%%%%%%%%%%%%%%%%%%%%%%%%%%%%%%%%%%%%%%%%%%%%%%%%%

(E)  Set   $H(a)=h_a(x_0)=T_b^{-1}\big(h(x_0 a)\big)$,  $w'=T_b^{-1}(w)$,    $d_0(w')=dist(w',\partial D_0^-)$  and   $d^0(a)=d_0\left(T_b^{-1}\big(h(x_0 a)\big)\right)$, where
$b=h(a)$. It is straightforward to check  that  $H$ and $d^0$ are continuous function with respect  to  $a \in \mathbb{T}$. Hence  there is $s_0>0$  such that   $B(H(a),s_0)\subset D_0^-$ and therefore  we conclude

(E0) If $h$ satisfies the hypothesis $\mathbb{U}_{qc}^0$, there is a constant $s_0>0$ which does not depend of $a$  such that     $B(h(x_0 a),s_0)\subset D_b^-$, $a \in \mathbb{T}$.

%%%%%%%%%%%%%%%%%%%%%%%%%%%%%%%%%%%%%%%%%%%%%%%%%%%%%%%%%%%
Recall that  for   $a=e^{i\alpha}  \in \mathbb{T}$, we  let $\phi=\phi_a$
be the conformal mapping of  $\mathbb{U}$  onto   $U_a$   such that  $\phi_a(0)=x_0 e^{i\alpha}$  and $F=F(h)=F_a= h\circ \phi_a$.
Let  $b\in  \partial D$,    $w-b=\epsilon n_b$, $\epsilon \leq \epsilon_0$. Then $w\in D_b^-$.

%Set   $a=h^{-1}(b)$, $z=h^{-1}(w)$,  $z'=\phi_a^{-1}(z)$,    $z=h^{-1}(w)$, $d(z')= 1-|z'|$, $r'=|z'|$, $d_a(z)=dist(z,\hat{U}_a)$,
%$d_a^+(z)=dist(z,U_a^+)$  $U_a^-=h^{-1}(D_b)$,
%and $d_b(w)=dist(w,\partial D_b^-)$.
\begin{eqnarray}
 \mbox{Set}\,  \,   a&=&h^{-1}(b), \, z=h^{-1}(w),\,  z'=\phi_a^{-1}(z), \,  d(z')= 1-|z'|,\,  r'=|z'|, \nonumber\\
 d_a(z)&=&dist(z,\partial \hat{U}_a),\,   d_b(w)=dist(w,\partial D_b^-) \quad   \mbox{and} \quad  d_0(w')=dist(w',\partial D_0^-). \nonumber
\end{eqnarray}
Let  $H_b$ be the half plane  which  contains $\hat{D}_b$ and touches $D$ at  $b$.  Since  $F_a$ maps  $\mathbb{U}$ into $H_b$,  $(F_a(z') -b,n_b)$   is non-negative in
$z'\in \mathbb{U}$, and  by a version of Harnack's estimate,   $|F_a(z') - b|\geq s^0 (1-r')$, i.e.     $|w - b|\geq s^0 d(z')$.
Hence   (1)   $|w-b|\succeq d(z')$.
% $d \Lambda_f \succeq d_*$\approx    \preceq

From Lemma \ref{le_Lyp1}  (the geometric property of the domain $D_b^-$),  we find $|w-b|\preceq d_b(w)$  and therefore
\begin{itemize}
\item [(3)]   \quad   $d_b(w)\succeq d(z')$.
\end{itemize}
Set $\epsilon_0 := \min\{\varepsilon_0^-,  \epsilon^0\}$ and   $D(\epsilon_0):= \{w\in D : d(w,\partial)\geq \epsilon_0\}$. Then there is $r_0,r_1\in(0,1)$  such that  $D(\epsilon_0)\subset h(B(r_0))$  and   $r'=|z'|\geq r_1$
implies  $|z|\geq r_0$.
%$B(r_0)\subset h(B(r_1))$.

By Theorem  \ref{thm.space1} (note here that  the estimate in this theorem depends on  $R_0$ and it is independent of $h$)  and (E0), we conclude

(E1)  there is  $s_1>0$  which is independent of $z'$  such that   $d_b(w)\geq s_1 d(z')$,  $r'=|z'|\geq r_1$.

Recall that  $d'_b(w)= dist (w, Y_a)$, where $Y_a:=h(\hat{U}_a)$.
We now estimate $\Lambda_h(z)$.  Using the fact that $h(\hat{U}_a)= Y_b \supset D_b^-$, we find  first that $d'_b(w) \geq  d_b(w)$  and  therefore
$$\Lambda_h(z)\succeq \frac{d'_b(w)}{d_a(z)}\geq \frac{d_b(w)}{d_a(z)} .$$

Since $\phi_a(\mathbb{U})=\hat{U}_a$  and  $\hat{U}_a$  is a Lyapunov domain  of a fixed shape,  $d_a(z) \approx  d(z')$.

Combining this with (3),  we conclude   $\lambda_h(z) \approx \Lambda_h(z)\geq s_2>0$, where $s_2>0$  is a constant. Hence,  using (E1), we conclude

(F)   $\lambda_h(z) \approx \Lambda_h(z)\geq s_2>0$, $|z|\geq r_0$,  where $s_2>0$  is a constant  independent of $z$.

It is clear that there  is a constant   $s_3>0$  such that

(F1) $\lambda_h(z) \approx \Lambda_h(z)\geq s_3>0$, $z\in B(r_0)$.

By (F) and (F1),   there  is a constant   $s_4>0$  such that  $\lambda_h(z) \approx \Lambda_h(z)\geq s_4>0$, $z\in \mathbb{U}$.

Hence it is readable  that  $h$ is co-Lip  on  $\mathbb{U}$.

\begin{remark}  By   application   of   Proposition    \ref{Harprop1}  onto the  restriction of $h$   on  $\hat{U}_a$    one can also get a  proof    of   Theorem  \ref{thmain2}.
\end{remark}

\section{Further comments and related results}
%C. J.
We briefly discus  the connection  with
the Rad\'{o}-Kneser-Choquet theorem (shortly RKC-Theorem) and hyperbolic-harmonic mappings;  it will be the subject of further investigations.\\
 Quasiconformal euclidean-harmonic mappings are  bi-Lipschitz  with respect to the
quasi-hyperbolic metric, cf. \cite{napoc1,man} (Proposition \ref{c.gd} here).
It turns out that, as in the euclidean case, quasiconformal hyperbolic-harmonic mappings    are  bi-Lipschitz    with respect to the
hyperbolic metric, cf.  Wan \cite{wan1} and of Markovic \cite{vmar1}.\\
%For related results about quasiconformal harmonic mappings with respect to the hyperbolic metric we refer to the paper of Wan \cite{wan1} and of Markovi\'c \cite{vmar1}.
%%%%% e Ins 1  the result of Li and Tam  (cf. [12]in \cite{marInv}) that every diffeomorphism of $\mathbb{S}^2$ admits a harmonic quasiisometric extension
Very recently,  concerning the initial   Schoen Conjecture (and more generally  the Schoen-Li-Wang conjecture) Markovic made a  major breakthrough.
In \cite{vmar2},   Markovic used the result of Li and Tam  that every diffeomorphism of $\mathbb{S}^2$ admits a harmonic quasiisometric extension to show that every quasisymmetric homeomorphism of the circle $\partial \mathbb{H}^2$ admits a harmonic quasiconformal extension to the hyperbolic plane $\mathbb{H}^2$. This proves the initial   Schoen Conjecture.

In particular, concerning complex valued harmonic functions,  Kalaj and  the second author,
%Mateljevi\'c
shortly KM-approach, study lower bounds of the Jacobian, cf. \cite{rckm-fil,MMfilInv} and references cited there.
The corresponding results for harmonic maps between surfaces were
 previously obtained by  Jost and Jost-Karcher  \cite{jost2,jost}.
We refer to this result as the JK- result (approach).
%%%%%
%{mm.unpub2} {aaa} harmonic mappings to
%%%%%
G. Alessandrini and  V. Nesi  prove necessary and   sufficient criteria of invertibility for planar harmonic mappings
which generalize a classical result of H. Kneser, also known as
the Rad\'{o}-Kneser-Choquet theorem (RKC-Theorem), cf. \cite{ales.nes}.
%Recently   Iwaniec- Onninen   have communicated a new   proof of  the RKC theorem
% cf. \cite{IwOn}. We refer to this    communication shortly as IwOn-approach.  It seems that there is some overlap between  KM- results  with  \cite{IwOn} and   \cite{jost2,jost}.
Note only here that in the planar case the JK- result  is   reduced to Theorem RKC.
Kalaj \cite{kal.studia} also  has extended  the Rado-Choquet-Kneser
theorem to mappings between the unit circle  and Lyapunov  closed curves  with Lipschitz boundary data and essentially
positive Jacobian at the boundary (but without restriction on the
convexity of the image domain). The proof is based on
the  extension of the Rado-Choquet-Kneser theorem by Alessandrini and
Nesi  \cite{ales.nes2} and  an approximation scheme is used in it.
 Motivated by an approach described in Kalaj's  Studia paper\cite{kal.studia}  and using the continuity of so called $E$-function,  the second author  found  a new proof of  Kalaj's  result, cf. \cite{rckm-fil,MMfilInv}.

{\it Acklowedgment}.
We are indebted  to  M.  Svetlik for helping us in preparation this manuscript. In particular we thank    him for making the Figure 1.
We are indebted  to the   referee,  N. Mutavd\v zi\'c   and D. Kalaj  for useful comments which improved the exposition.

{}
\end{document}